\documentclass[british]{amsart}

\usepackage{mathptmx}

\usepackage[T1]{fontenc}
\usepackage{textcomp}
\usepackage{xcolor}
\usepackage{babel}
\usepackage{amstext}
\usepackage{amsthm}
\usepackage{amssymb}
\usepackage{tikz-cd}
\usepackage{mathrsfs} 
\usepackage[pdfusetitle,
bookmarks=true,bookmarksnumbered=false,bookmarksopen=false,
breaklinks=false,pdfborder={0 0 1},backref=false,colorlinks=false]
{hyperref}
\usepackage{comment}
\makeatletter
\numberwithin{equation}{section}
\numberwithin{figure}{section}
\theoremstyle{plain}
\newtheorem{thm}{\protect\theoremname}[section]
\theoremstyle{plain}
\newtheorem{cor}[thm]{\protect\corollaryname}
\theoremstyle{definition}
\newtheorem{defn}[thm]{\protect\definitionname}
\theoremstyle{plain}
\newtheorem{lem}[thm]{\protect\lemmaname}
\theoremstyle{definition}
\newtheorem{example}[thm]{\protect\examplename}
\theoremstyle{remark}

\newtheorem{rem}[thm]{\protect\remarkname}
\theoremstyle{plain}
\newtheorem{prop}[thm]{\protect\propositionname}
\newtheoremstyle{claimstyle}
{0.8\baselineskip} 
{0.8\baselineskip} 
{\itshape}         
{}                 
{\itshape}         
{.}                
{.5em}             
{}                 

\theoremstyle{claimstyle}

\newtheorem*{claim*}{Claim}

\theoremstyle{plain}

\theoremstyle{plain}
\theoremstyle{plain}

\addto\captionsbritish{\renewcommand{\conjecturename}{Conjecture}}
\addto\captionsenglish{\renewcommand{\conjecturename}{Conjecture}}
\providecommand{\conjecturename}{Conjecture}

\@ifundefined{date}{}{\date{}}

\makeatother

\addto\captionsbritish{\renewcommand{\corollaryname}{Corollary}}
\addto\captionsbritish{\renewcommand{\definitionname}{Definition}}
\addto\captionsbritish{\renewcommand{\examplename}{Example}}
\addto\captionsbritish{\renewcommand{\lemmaname}{Lemma}}
\addto\captionsbritish{\renewcommand{\propositionname}{Proposition}}
\addto\captionsbritish{\renewcommand{\remarkname}{Remark}}
\addto\captionsbritish{\renewcommand{\theoremname}{Theorem}}
\addto\captionsenglish{\renewcommand{\corollaryname}{Corollary}}
\addto\captionsenglish{\renewcommand{\definitionname}{Definition}}
\addto\captionsenglish{\renewcommand{\examplename}{Example}}
\addto\captionsenglish{\renewcommand{\lemmaname}{Lemma}}
\addto\captionsenglish{\renewcommand{\propositionname}{Proposition}}
\addto\captionsenglish{\renewcommand{\remarkname}{Remark}}
\addto\captionsenglish{\renewcommand{\theoremname}{Theorem}}
\providecommand{\corollaryname}{Corollary}
\providecommand{\definitionname}{Definition}
\providecommand{\examplename}{Example}
\providecommand{\lemmaname}{Lemma}
\providecommand{\propositionname}{Proposition}
\providecommand{\remarkname}{Remark}
\providecommand{\theoremname}{Theorem}

\newcommand{\Fix}{\operatorname{Fix}}

\newcommand{\calF}{\mathcal{F}}

\newcommand{\calP}{\mathcal{P}}
\newcommand{\calJ}{\mathcal{J}}

\newcommand{\calL}{\mathcal{L}}

\newcommand{\calD}{\mathcal{D}}

\newcommand{\calO}{\mathcal{O}}
\newcommand{\calS}{\mathcal{S}}
\newcommand{\calB}{\mathcal{B}}
\newcommand{\calG}{\mathcal{G}}

\newcommand{\bd}{\partial_\infty}
\newcommand{\cat}{\mathrm{CAT}}
\newcommand{\Ch}{\operatorname{Ch}}

\newcommand{\chinf}{\operatorname{ch}(X^\infty)}

\newcommand{\bdinf}{X^\infty}
\newcommand{\Aut}{\operatorname{Aut}}

\newcommand{\iso}{\operatorname{Isom}}

\newcommand{\SL}{\operatorname{SL}}

\newcommand{\res}{\operatorname{Res}}

\newcommand{\Opp}{\operatorname{Opp}}

\newcommand{\Sp}{\operatorname{Sp}}

\newcommand{\newcomment}[4]{%
	\newcounter{#2counter}
	\expandafter\newcommand\csname #1\endcsname[1]{%
		\refstepcounter{#2counter}%
		{\color{#4}(#3\arabic{#2counter})}\marginpar{\scriptsize\raggedright\textbf{\color{#4}(#2 \arabic{#2counter}):} ##1}%
}}

\definecolor{copperrose}{rgb}{0.6, 0.4, 0.4}

\global\long\def\explain#1#2{\underset{\underset{\mathclap{#2}}{\downarrow}}{#1}}%

\newcomment{clb}{Corentin}{c}{teal}
\newcomment{iv}{Itamar}{i}{orange}
\newcomment{shl}{Shivam}{s}{purple}

\begin{document}
	\selectlanguage{british}%
	\global\long\def\res{\!\restriction}%
	
	\global\long\def\actson{\curvearrowright}%
	
	\global\long\def\Ad{\text{Ad}}%
	
	\global\long\def\Prob#1{\text{Prob}\left(#1\right)}%
	
	\global\long\def\Ch#1{\text{Ch}\left(#1\right)}%
	
	\global\long\def\Tr#1{\text{Tr}\left(#1\right)}%
	
	\global\long\def\Trleq#1{\text{Tr}_{\leq1}\left(#1\right)}%
	
	\global\long\def\normalizer#1#2{\text{N}_{#1}\left(#2\right)}%
	
	\global\long\def\FC#1{\text{FC}\left(#1\right)}%
	
	\global\long\def\Sub#1{\text{Sub}\left(#1\right)}%
	
	\global\long\def\conn#1{#1^{0}}%
	
	\global\long\def\explain#1#2{\underset{\underset{\mathclap{#2}}{\downarrow}}{#1}}%
	
	\global\long\def\Ind#1#2{\text{Ind}_{#2}^{#1}}%
	
	\global\long\def\Res#1#2{\text{Res}_{#2}^{#1}}%
	
	\global\long\def\N{\mathbb{N}}%
	\global\long\def\Z{\mathbb{Z}}%
	\global\long\def\Q{\mathbb{Q}}%
	\global\long\def\R{\mathbb{R}}%
	\global\long\def\C{\mathbb{C}}%
	\global\long\def\H{\mathbb{H}}%
	\global\long\def\T{\mathbb{T}}%
	\global\long\def\P{\mathbb{P}}%
	\global\long\def\K{\mathbb{K}}%
	
	\global\long\def\bG{\mathbf{G}}%
	\global\long\def\bP{\mathbf{P}}%
	\global\long\def\bI{\: \mathbf{I} \: }%
	\global\long\def\nbI{\: {\not\mathbf{I}} \: }%
	\global\long\def\BH{\mathcal{B}(\mathcal{H})}%
	\global\long\def\cG{\mathcal{G}}%
	
	\global\long\def\sl#1{\mathrm{SL}(#1)}%
	
	\global\long\def\slz{SL_{n}\left(\mathbb{Z}\right)}%
	\global\long\def\slr{SL_{n}\left(\mathbb{R}\right)}%
	
	\global\long\def\conn#1{#1^{0}}%

	\title[selflessness, mif and opposition in groups acting on exotic buildings]{selflessness, mif and opposition in groups acting \\ 
    on exotic buildings}
	
	\begin{abstract}
		We prove that groups acting freely and cocompactly on (possibly exotic) affine buildings of type $\tilde{A}_2$ and $\tilde C_2$ are mixed-identity free and have selfless reduced $C^*$-algebras. These results follow from a strong form of ping-pong dynamics  that we call 'transversal contractivity'.  Our main geometric result regards domesticity properties of elements in the associated polygons at infinity: we prove that, in our context, the opposite geometry of a hyperbolic element is topologically large.
	\end{abstract}

	\author{Corentin Le Bars, Elyasheev Leibtag, Itamar Vigdorovich}
	\thanks{C.LB.\ is supported by ERC Advanced Grant NET 101141693. 
		E.L.\ is supported by the FWO and the F.R.S. FNRS under the Excellence of Science (EOS) program (project ID 40007542).
		I.V.\ was supported by NSF postdoctoral fellowship grant DMS-2402368.}
	\maketitle
	\setcounter{tocdepth}{1}
	
	\section{Introduction}
	
	Let \(X\) be a thick locally finite affine building and let
	\(G=\operatorname{Isom}(X)\).  The classical examples arise from
	semisimple algebraic groups over local fields, but in dimension \(2\) there
	are many affine buildings which are not of Bruhat--Tits type
	\cite{ronan86}.  These are called \emph{exotic} buildings.  Groups acting
	geometrically on exotic buildings are often non-linear  \cite{bader_caprace_lecureux19}, but they retain higher-rank features such as
	property~(T), quasi-isometric rigidity and normal subgroup theorems
	\cite{KleinerLeeb1997,Oppenheim2025,lecureux_witzel26}.

 In this paper, we establish several new results for groups acting on rank \(2\) buildings.  To avoid unnecessary technicalities, we state our main results first for cocompact lattices. In Section~1.4, we describe the precise class of groups to which our results apply.

    \subsection{Mixed identities, and the  permutational Boone--Higman conjecture}

Let \(\Gamma\) be a group and fix a generator  $x$  for the group of integers $\Z$.  A
	 word \(w=w(x)\in\Gamma*\langle x\rangle\) is viewed as a  formula  in the variable $x$ with coefficients in $G$. $w$ is called  a \emph{mixed identity} of $\Gamma$
	if \(w(\gamma)=e\) for every \(\gamma\in\Gamma\). For example,  $z\in \Gamma$ is in the center of $\Gamma$ if and only if $w(x)=zxz^{-1}x^{-1}$ is a mixed identities.   A group with no mixed
	identity (other than the trivial element of $\Gamma * \Z$) is called \emph{mixed identity free}, or \emph{MIF}.
	
	\begin{thm}\label{Thm:MIF}
		\label{thm:intro-building-MIF}
		Let \(X\) be a locally finite affine building, and let 
		\(\Gamma<\operatorname{Isom}(X)\) be a type-preserving cocompact lattice. Assume that either $X$ if of type \(\widetilde A_2\), or that $X$ is of type \(\widetilde C_2\) and the action $\Gamma\actson X$ is free. Then $\Gamma$ is MIF.
	\end{thm}
    The additional assumption in type \(\widetilde C_2\) cannot be omitted, as can already be seen in the classical case of the symplectic group $\Sp_4$, see \cite{Tomanov1985,AvniGelander2025}.\\

   Theorem~\ref{Thm:MIF} is of particular interest in light of the \emph{permutational Boone--Higman conjecture} \cite{BFFHZ26}. Belk, Fournier-Facio, Hyde, and Zaremsky proved that every finitely presented simple MIF group satisfies the conjecture. They note that most known classes of finitely presented simple groups are MIF, while specifically pointing out the simple exotic lattices of Titz Mite--Witzel \cite{TitzMiteWitzel2025} as a class for which MIF was not yet known. Theorem~\ref{Thm:MIF} establishes MIF for these \(\widetilde C_2\)-lattices and hence verifies the permutational Boone--Higman conjecture for them.

     We expect that a quantitative version of Theorem~\ref{Thm:MIF} should also hold, despite the fact that the standard methods available in the linear setting do not apply here. Quantitative aspects of MIF have attracted considerable attention in recent years; see \cite{Bradford2024,BradfordSchneiderThom2024,vigdorovich25,AvniGelander2025,BradfordSisto2026}. For further results on the MIF property  see \cite{Tomanov1985,HullOsin2016,Jacobson2021,EtedadialiabadiGaoLeMaitreMelleray2021,BodirskySchneiderThom2025,HydeLodha2025,Rybak2026}.

    \subsection{Selfless reduced $C^*$-algebras.}
    The next property we discuss regards $C^*_r(\Gamma)$, the reduced $C^*$-algebra of the group $\Gamma$. Recall that $C^*_r(\Gamma)$ is defined as the norm-closure of $\C[\Gamma]$ embedded inside $\mathcal{B}(\ell^2(\Gamma))$ via the left-regular representation. 
    
   The notion of a selfless \(C^*\)-algebra was introduced by Robert \cite{`2023selfless}. In the case of a reduced group $C^*$-algebra $C^*_r(\Gamma)$, selflessness means that the natural embedding $C^*_r(\Gamma)\hookrightarrow C^*_r(\Gamma*\Z)$ is existential in the sense of (continuous) model theory. Interestingly, $\Gamma$  is MIF precisely when the embedding $\Gamma\hookrightarrow \Gamma*\Z $ is existential. Thus selflessness can  be viewed as a \(C^*\)-algebraic analogue of the MIF property for groups, as first discovered in \cite{amrutam_gao_kunnawalkam-elayavalli_patchell25}. Although there are some concrete implications connecting the two notions, this relationship remains largely an analogy. 
   
   Selflessness implies several of the most important regularity properties in the structure theory of \(C^*\)-algebras. The property attracted considerable attention following the breakthrough work \cite{amrutam_gao_kunnawalkam-elayavalli_patchell25}, which established strict comparison for \(C_r^*(\mathbb F_n)\) via selflessness. Notable classes of groups whose reduced \(C^*\)-algebras are selfless include linear groups with trivial amenable radical \cite{vigdorovich25,Vigdorovich2026Linear}, acylindrically hyperbolic groups with trivial finite radical \cite{amrutam_gao_kunnawalkam-elayavalli_patchell25,Ozawa2025,Yang2025}, and several further classes \cite{KunnawalkamPatchellTeryoshin2025,GaoKunnawalkamPatchellTeryoshin2026,BasuFlores2026,ArzhantsevaFinnSell2026,flores2025selfless}.

	\begin{thm}
		\label{thm:intro-selflessness}
		Let \(X\) be a locally finite affine building, and let 
		\(\Gamma<\operatorname{Isom}(X)\) be a type-preserving cocompact lattice. Assume that either $X$ if of type \(\widetilde A_2\), or that $X$ is of type \(\widetilde C_2\) and the action $\Gamma\actson X$ is free. Then $C^*_r(\Gamma)$ is selfless.
	\end{thm}
    We recall that any selfless $C^*$-algebra is simple \cite[Theorem 3.1]{robert25}.
	Simplicity of $C^*_r(\Gamma)$ was very recently established \cite{ciobotaru_le-bars26}. It is open whether simplicity implies selflessness for reduced group $C^*$-algebras \cite[Problem XCI]{schafhauser2025nuclear}

    \subsection{Opposition and transversely contractive dynamics}

The key idea underlying both Theorem~\ref{Thm:MIF} and Theorem~\ref{thm:intro-selflessness} is the following dynamical group property, which is a variant of Ozawa's property \(P_{PHP}\) and of the Powers-type notion of de la Harpe \cite{deLaHarpe1985,Ozawa2025}.

\begin{defn}\label{def:P-star-PHP}
An action \(\Gamma\curvearrowright S\) is \emph{transversely contractive}
if, for every finite set
\(F\subseteq\Gamma\setminus\{e\}\), there exist
\(\gamma\in\Gamma\) and two disjoint subsets \(U^+,U^-\subseteq S\) such that:
\begin{enumerate}
\item
\(
\gamma(S\setminus U^-)\subseteq U^+ 
\) (note that this implies \(\gamma^{-1}(S\setminus U^+)\subseteq U^-\));
\item for every \(f\in F\),
\(
(U^+\cup U^-)\cap f(U^+\cup U^-)=\varnothing.
\)
\end{enumerate}
We say that a group is \emph{transversely contractive} if it admits such an action.
\end{defn}

This is an intrinsic group property: if \(\Gamma\) admits a transversely contractive action, then its left-regular action is transversely contractive as well, see Lemma~\ref{lem:pullback-left-regular-transversal}. Moreover, transversal contractivity implies MIF as well as selflessness of the reduced $C^*$-algebra, see Section~\ref{sec:P-star-from-strong-nondomesticity}. Thus, our goal is to establish transversal contractivity.

Establishing this property for groups acting on hyperbolic spaces is straightforward. One can choose \(\gamma\) to be a `sufficiently generic' loxodromic element, and take \(U^+\) and \(U^-\) to be sufficiently small neighborhoods of the attracting and repelling points of \(\gamma\) in the Gromov boundary.

In higher rank, one naturally considers the action on the flag space. A loxodromic element \(\gamma\) has an attracting point, but the repelling set is now a union of Schubert cells rather than a single point. These cells are large: in the classical case, the dimension is strictly greater than half the total dimension of the flag space. As a result, taking \(U^-\) to be a neighborhood of this repelling set, as is necessary for condition~(1), would lead to unavoidable intersections between \(U^-\) and \(fU^-\), making condition~(2) problematic.

We overcome this difficulty by splitting the root system into two blocks of equal size (see Figures~\ref{figure B(C) A2} and \ref{figure B(C) C2-G2}), which results in a smaller choice of \(U^-\). This idea is inspired by \cite{deLaHarpe1985}.\footnote{We note that, as mentioned at the end of the manuscript, there is an error in the argument of \cite{deLaHarpe1985}. Nevertheless, it is this argument that inspired the present construction.}\\

There is yet another fundamental obstacle in the higher-rank setting, namely domesticity. The following notions are explained in further detail in Section \ref{section:limit_set}.

Let \(X^\infty\) be the boundary  at infinity of a two-dimensional affine building \(X\). It is a rank-two spherical building, and  let $\mathcal F $ denote the set of chambers, endowed with the opposition relation.  

An automorphism \(g\) of \(X^\infty\) is called \emph{domestic} if its opposite geometry
$$
\mathcal O_{\mathcal F}(g)
:=
\{C\in\mathcal F\mid gC\text{ is opposite to }C\},
$$
is empty.

We extend this notion to group actions as follows. Say that an action by isometries  \(\Gamma\curvearrowright X\) is \emph{strongly non-domestic} (on the flag space) if, for every finite
\(F\subseteq\Gamma\setminus\{e\}\),

$$
\Lambda_{\mathcal F}(\Gamma)
\cap
\bigcap_{g\in F}\mathcal O_{\mathcal F}(g)
\neq\varnothing,
$$
where \(\Lambda_{\mathcal F}(\Gamma)\subseteq\mathcal F\) is the flag limit set.

 We use the notion of a group of general type (see Definition~\ref{def general type}), which is  a geometric notion introduced in \cite{ciobotaru_le-bars26} that is meant to capture Zariski density, while remaining available in exotic case.

\begin{thm}[Strong non-domesticity implies transversal contractivity]\label{Thm:intro-str_non-dom_trans_cont}
Let \(X\) be a locally finite irreducible\footnote{ The reducible case $\widetilde A_1\times \widetilde A_1$ is much simpler as it is the action on a product of trees.} affine building of any type. Let \(\Gamma<\operatorname{Isom}(X)\) act properly and type-preservingly. Assume that
\(\Gamma\) is of general type and that the action
\(\Gamma\curvearrowright\mathcal F\) is strongly non-domestic. Then the
action \(\Gamma\curvearrowright\mathcal S\) is transversely contractive, where
\(\mathcal S\) is the space of chambers in the \(\widetilde A_2\) case and the space of panels in the \(\widetilde C_2\) and \(\widetilde G_2\) cases.
\end{thm}

Cocompact lattices are groups of general type with full flag limit set by \cite[Proposition~6.1]{ciobotaru_le-bars26}. Theorem ~\ref{Thm:intro-str_non-dom_trans_cont} therefore applies, but includes additional examples, thus extending Theorem \ref{Thm:MIF} and Theorem \ref{thm:intro-selflessness}.

Note that Theorem~\ref{Thm:intro-str_non-dom_trans_cont} is interesting also in the non-exotic case, i.e. for lattices in $\SL_3(\Q_p)$. An approach that is very similar to ours, also based on contracting dynamics on the flag variety and that is actually outlined in \cite[\S4]{deLaHarpe1985} would give that \(\operatorname{SL}_3(\Z)\) is transversely contractive. In contrast, \(\operatorname{SL}_n\) for \(n\geq4\) has many domestic elements (see Example~\ref{ex:non-domestic}), and we leave open whether lattices in these groups are transversely contractive.\\

Our goal is therefore to prove that the rank-two groups appearing in Theorem~\ref{Thm:intro-str_non-dom_trans_cont} satisfy the strong non-domesticity property. To this end, we study the opposite geometry of elements. A central geometric theorem of the paper is that, in types \(\widetilde A_2\) and \(\widetilde C_2\), opposite geometry is topologically large. 

\begin{thm}[Dense opposite geometry]\label{thm_intro:dense_opp_geom}
Let \(g\) be a type-preserving hyperbolic isometry of a locally finite thick
affine building of type \(\widetilde A_2\) or \(\widetilde C_2\). Then
\(\mathcal O_{\mathcal F}(g)\) is open and dense in \(\mathcal F\).
In the \(\widetilde A_2\)-case, the same conclusion is valid under the weaker assumption that the set of \(g\)-fixed flags
\(\operatorname{Fix}_{\mathcal F}(g)\) has empty interior.
\end{thm}

Proposition~\ref{prop:A2_hyperbolic_no_open_fixed_flags} below shows that in the \(\widetilde A_2\)-case, being hyperbolic implies that \(\operatorname{Fix}_{\mathcal F}(g)\) has empty interior. Our proof in the $\tilde C_2$-case really uses the hyperbolicity assumption: we do not know if the conclusion of the lemma holds if \(\operatorname{Fix}_{\mathcal F}(g)\) is only assumed to have empty interior. 

Theorem~\ref{thm_intro:dense_opp_geom} immediately gives the following corollary.

\begin{cor}
Let \(X\) be a locally finite affine building of type
\(\widetilde A_2\) or \(\widetilde C_2\). Let
\(\Gamma<\operatorname{Isom}(X)\) be a  subgroup of general type acting properly, type-preserving, and with full flag limit set. If \(X\) is of type \(\widetilde C_2\), assume furthermore that every element is hyperbolic. Then
\(\Gamma\curvearrowright\mathcal F\) is strongly non-domestic.
\end{cor}

Theorem~\ref{thm_intro:dense_opp_geom} has independent interest: it provides a dynamical counterpart to the study of domestic automorphisms of spherical buildings, developed for instance by Parkinson, Temmermans, Thas, and Van Maldeghem
\cite{temmermans_thas_vanmaldeghem12,parkinson_temmermans_van-maldeghem15,parkinson_van-maldeghem19,parkinson_van-maldeghem24}. This line of work originated with Leeb's observation that non-trivial automorphisms of thick spherical buildings always move some residue to an opposite residue \cite[Sublemma~5.22]{leeb00}. It is also related to Phan theory, which studies the opposition geometry of elements in Moufang spherical buildings. Here the spherical building at infinity may be unrelated to any algebraic variety.

	\subsection*{Acknowledgments}
	C.LB.\ is supported by ERC Advanced Grant NET 101141693. 
	E.L.\ is supported by the FWO and the F.R.S. FNRS under the Excellence of Science (EOS) program (project ID 40007542).
	I.V.\ was supported by NSF postdoctoral fellowship grant DMS-2402368, and would like to thank Tianyi Zheng for fruitful discussions at the early stages of this project. We also thank Hendrik Van Maldeghem for helpful discussions with the second author which led to Lemma~\ref{lem:gq-local-fixed-incidence} and its proof.
	
\subsection*{Clarification of the use of AI}

An initial version of this paper contained all the results concerning
$\tilde{A}_2$-buildings, but only partial results concerning
$\tilde{C}_2$- and $\tilde{G}_2$-buildings. This version was developed
without any assistance from AI.

We identified the missing ingredients needed to extend the
results to $\tilde{C}_2$- and $\tilde{G}_2$-buildings: a characterization
of non-domestic elements and a meagerness statement for their domestic
loci, as in Theorem~\ref{thm:A2-open-dense-opposite-geometry}. We were aware of
the results of \cite{temmermans_thas_vanmaldeghem12} characterizing domestic elements in generalized quadrangles, but we had not managed to conclude the argument. We then used ChatGPT-5.6 Sol in order to prove Theorem~\ref{thm:gq-open-dense-opposite-geometry}: we asked whether an approach similar to that of the $\tilde{A}_2$-case, combined with the results and ideas present in
\cite{temmermans_thas_vanmaldeghem12} could be used to obtain the analogous statement for the $\tilde C_2$-case. Part of our input was Lemma~\ref{lem:gq-local-fixed-incidence}, of which we had a proof sketch. After some interaction, the model provided Lemma~\ref{lem:gq-common-neighbour-incidence} and Proposition~\ref{prop:gq-fixed-incidence-hyperbolic} which, combined with Lemma~\ref{lem:gq-local-fixed-incidence} and the outlined strategy, completed the proof of Theorem~\ref{thm:gq-open-dense-opposite-geometry}. We also tried to work on the $\tilde{G}_2$-case with the same model, but this did not turn out to be successful.

The authors subsequently revised the arguments and presentation
substantially and take full responsibility for the final mathematical
content of the paper.

	\section{Opposition in topological generalized polygons}

	\subsection{Affine and spherical buildings}
	
	We briefly present the notion of (discrete) buildings, following \cite{abramenko_brown08}. For the notion of Coxeter complexes, we refer to \cite[Chapter~3]{abramenko_brown08}. 
	\begin{defn}
		Let \(\Sigma\) be a Coxeter complex. A building modelled on \(\Sigma\) is a simplicial complex \(X\) covered by subcomplexes, called apartments, such that:
		\begin{itemize}
			\item every apartment is isomorphic to \(\Sigma\);
			\item every two simplices of \(X\) are contained in a common apartment;
			\item if \(\Sigma_1\) and \(\Sigma_2\) are apartments, then there exists
			an isomorphism \(\Sigma_1\to\Sigma_2\) fixing
			\(\Sigma_1\cap\Sigma_2\) pointwise.
		\end{itemize}
		The type of \(X\) is the type of the Coxeter complex \(\Sigma\). The building is called spherical, respectively affine, if \(\Sigma\) is a spherical, respectively affine, Coxeter complex.
	\end{defn} 
	
	The dimension of \(X\) is the dimension of the model Coxeter complex \(\Sigma\). Equivalently, if the Coxeter system has finite generating set \(S\), then $\dim X=|S|-1$. In this paper, we only consider affine buildings of dimension 2 and spherical buildings of dimension 1. Simplices of maximal dimension are called \emph{chambers}. Affine buildings of dimension 2 are of type $\tilde A_1 \times \tilde A_1$ (in which case $X$ is a product of simplicial trees),  $\tilde A_2 $, $\tilde C_2$ or $\tilde G_2$. 
	
	\subsection{Opposition}
	
	In a spherical building, two chambers \(C\) and \(D\) are called \emph{opposite} if their Weyl distance is the longest element \(w_0\) of the spherical Weyl group, i.e. they are contained in a common apartment, and in that apartment they correspond to opposite chambers of the spherical Coxeter complex. Two opposite chambers are contained in a unique common apartment. Two sets of chambers $V,V'$ are said to be \emph{totally opposite} if $x$ and $x'$ are opposite for any $x\in V,x'\in V'$
	
	The model spherical Coxeter complex carries its usual angular metric. Since every two points of the geometric realization of a spherical building are contained in an apartment, and since apartment changes are isometries on overlaps, these angular metrics glue to a well-defined metric on the spherical building. We denote it by \(d_T\), and call it the \emph{Tits} or \emph{angular} metric. Thus, a Tits ball is a ball for this metric. In rank two, apartments are metric circles of length \(2\pi\): for type \(I_2(n)\), each chamber has angular length \(\pi/n\).
	
	\subsection{Spherical building at infinity}
	
	An affine building inherits from $\Sigma$ a metric which makes it a $\cat$(0) space. Therefore $X$ has a \emph{visual} bordification, given by equivalence classes of rays, two rays being equivalent if they are at finite Hausdorff $d$-distance. We denote this bordification by $\overline{X} = X \cup \bd X$. The visual boundary $\bd X$ can be endowed with a natural topology and isometries of $X$ extend to homeomorphisms on the boundary.  When $X$ is locally finite (which is a standing assumption in this paper), $\bd X $ and $\overline X$ are compact.
	
	The Weyl chambers in $X$ induce a simplicial structure on $\bd X$, where each simplex is the boundary of a given Weyl chamber. With this structure, $\bd X$ becomes a  spherical building \cite[Propri\'et\'e~1.7]{parreau00}. This building is called the \emph{spherical building at infinity} and denoted by $\bdinf$. The set of chambers of $\bdinf$ is denoted by $\chinf$.  Since we will be working on 2-dimensional buildings, the spherical building at infinity will be a graph, and $\chinf$ is the set of edges of this graph.

	\subsection{Generalized polygons}\label{section:n-gon}
	The class of rank-2 spherical buildings is canonically identified with the class of generalized polygons. As these are the only buildings we shall consider and the vocabulary of $n$-gons is better suited for our purpose, let us briefly present this notion. A $2$-dimensional incidence structure is a triple $\mathcal{G}=(\calP,\calL,\mathcal{F})$ where $\calP$ and $\calL$ are disjoint sets and $\mathcal{F}$ is a subset $\mathcal{F}\subset \calP\times \calL$ corresponding to an incidence relation. We write $p\bI L$ whenever $(p, L) \in \calF$. Elements of $\calP$ are called points, elements of $\calL$ are called lines, and elements of $\mathcal{F}$ are called flags or chambers.
	We denote by $\pi_\calP : \calF \to \calP$ and $\pi_\calL : \calF \to \calL$ the natural projections from flags to points and lines. 
	
	To any incidence structure of rank $2$, one associates its incidence graph   whose vertex set is $V=\calP\sqcup\calL$, and edge set $E=\mathcal{F}$. The vertices and edges of this graph is equipped with the path metric, denoted by $d$.
	
	A \emph{generalized $n$-gon} is an incidence structure $\cG$ whose incidence graph has diameter $n$ and such that, whenever $x,y\in V$ satisfy $d(x,y)<n$, there is a unique geodesic joining $x$ and $y$. A \emph{projective plane} is a generalized $3$-gon. Generalized $4$-gons, resp. $6$-gons are also called generalized quadrangles, resp. generalized hexagons. 
	A thick spherical building of type \(I_2(n)\) determines a thick generalized \(n\)-gon by declaring the two types of vertices to be points and lines, and by taking incidence to be adjacency in the building. Conversely, the incidence graph of a thick generalized \(n\)-gon is a rank \(2\) spherical building of type \(I_2(n)\). This identification is
	canonical up to exchanging the two types. As the spherical building at infinity of an affine building of dimension 2 has rank 2, it is canonically identified with an $n$-gon. We freely use this identification. In particular, the spherical building at infinity of an affine building of type $\tilde{A}_2 $ (resp. of type $\tilde{C}_2 $, $\tilde{G}_2 $) is a projective plane (resp. a $4$-gon, $6$-gon).

	Let $\mathcal{G}=(\calP,\calL,\mathcal{F})$ be a generalized polygon. For $p\in\calP$, we denote by $p^{\bI}$ the \emph{pencil of lines} at $p$, namely the set of all lines incident to $p$. Similarly, for $L\in\calL$, we denote by $L^{\bI}$ the \emph{point row} of $L$, namely the set of all points incident to $L$.
	
	Two points $p,q\in\calP$ are said to be \emph{collinear} if there exists a line $L\in\calL$ such that $p\bI L\bI q$; in this case we write $p\perp q$. The set of all points collinear with $p$ is
	\[
	p^\perp:=\{q\in\calP\mid p\perp q\}
	=\bigcup_{L\in p^{\bI}}L^{\bI}.
	\]
	
	Dually, two lines $L,L'\in\calL$ are said to be \emph{concurrent} if there exists a point $p\in\calP$ such that $L'\bI p\bI L$; in this case we write $L\perp L'$. The set of all lines concurrent with $L$ is
	\[
	L^\perp:=\{L'\in\calL\mid L\perp L'\}
	=\bigcup_{p\in L^{\bI}}p^{\bI}.
	\]
	
	We extend the incidence notation to subsets in the natural way. For
	$A\subseteq\calP$ and $B\subseteq\calL$, set
	\[
	A^{\bI}:=\{L\in\calL\mid L\bI p\text{ for some }p\in A\},
	\qquad
	B^{\bI}:=\{p\in\calP\mid p\bI L\text{ for some }L\in B\}.
	\]
	Thus, the operation $E\mapsto E^{\bI}$ passes from one type to the other. For a subset $E$ contained in either $\calP$ or $\calL$, we further write
	\[
	E^\perp:=(E^{\bI})^{\bI},
	\]
	so that $E^\perp$ is again a subset of the same type as $E$. 
	
	\subsection{Topological $n$-gons}
	We will consider generalized polygons arising as the spherical building at infinity of model spaces. Endowed with the cone topology, they become topological generalized polygons in the sense of \cite{kramer02}. We briefly present this topology, and we refer to \cite[\S3.2]{rousseau23} and \cite{ciobotaru_muhlerr_rousseau_20} for the details, where the more general case of masures is treated; see \cite[Remark~3.4]{ciobotaru_muhlerr_rousseau_20}. 
	
	Let $X$ be an affine building of dimension 2 and let $\calG = (\calP, \calL, \calF)$ be the $n$-gon arising as the spherical building at infinity of $X$. For any chamber $x \in \calF$, let $\xi_x \in \bd X $ be the spherical barycenter of $x $ for the Tits metric. Note that this choice of barycenters is preserved under the action of $\Aut(X)$. Let $o \in X$ be a point. For any boundary point $\xi \in \bd X$, we denote by $r^\xi : [0, +\infty)\to X$ the (unit-speed) geodesic ray in the class of $\xi$ with base point $o$. A standard open neighborhood of $x \in \calF $ is given by
	\[U_o(t, x) = \{ x' \in \calF \mid [o, r^{\xi_x}(t) ] \subseteq [o, \xi_{x'}) \},\]
	where $t >0$. 
	The topology generated by this set is called the cone topology on $\calF$. A similar construction can be done to give a topology on the set of points $\calP$ and to the set of lines $\calL$.

	Let $\calJ= \calP$ or $\calL$ be the set of points or the set of lines. We denote by $\pi_{\calJ} : \calF \to \calJ$ the map that associates to a flag its face of type $\calJ$. 
	Below we gather some properties of this topology. 
	\begin{prop}\label{prop topological spherical building}
		Let $\calJ\in \{\calP, \calL, \calF\}$ be a type. 
		\begin{enumerate}
			\item The topology on $\calJ$ as defined above does not depend on the choice of the basepoint $o \in X$, and is compact and metrizable. \label{prop top 1}
			\item Endowed with these topologies, the spherical building at infinity $X^\infty$ is a compact topological $n$-gon in the sense of \cite[Definition~7.3]{kramer02}. \label{prop top 5}
		\end{enumerate}
	\end{prop}
	
	\begin{proof}
		Statement \ref{prop top 1} follows from considerations valid in any locally compact  $\cat$(0) spaces. Statement \ref{prop top 5} is a direct consequence of \cite[Proposition~7.5]{kramer02}.
	\end{proof}
	
	We shall use the following fact.
	\begin{lem}\label{lem topo prop n gons}
		Let $\cG = (\calP, \calL, \calF)$ be a topological $n$-gon. Then the maps $\pi_\calP: \calF \to \calP$ and $\pi_\calL : \calF \to \calL$ are locally trivial bundles, hence continuous surjective and open. If $U$ is an open set of points, then the set $U^{\bI}$ of lines that meet $U$ is open (and dually for lines).
	\end{lem}
	\begin{proof}
		This is the content of \cite[Proposition~2.1.8 \& Proposition~2.1.9]{kramer90}. 
	\end{proof}

	Throughout, we always assume that generalized $n$-gons are thick, i.e. every point row and pencil of lines contains at least 3 elements. By an automorphism of a topological generalized $n$-gon we mean a homeomorphism preserving the incidence relation. We denote by $\Aut(\cG)$ the group of automorphisms of $\cG$. 
	
	\subsection{Domesticity}
	
	Let $\calG = (\calP, \calL, \calF)$ be an $n$-gon arising as the spherical building at infinity of $X$, endowed with the cone topology. 
	Two chambers $x $ and $y $ in $\calF$ are \emph{opposite} if they correspond to edges of maximal distance  in the associated incidence graph.  Any pair of opposite chambers $x^+, x^-$ belongs to a unique apartment $A^\pm$. We denote by $\calO(x)$ the set of chambers that are opposite  to $x$, which we refer to as the \emph{big cell} of $x$. We denote by $\mathcal{D}(x)= \calF - \calO(x)$ the set of chambers not opposite to $x$. 
	
	\begin{defn}\label{def:domestic}
		
		An element $\gamma \in \Aut(\calG)$ is \emph{domestic} if it maps no chamber to an opposite chamber. 
	\end{defn}
	
	\begin{rem}
		If $\calG$ is a Moufang $n$-gon, then by Tits' classification $\Aut(\calG)$ admits a BN-pair, where $B$ is the stabilizer of a chamber $x$, $N$ the normalizer of an apartment containing $x$, and $W=N/B\cap N$ the Weyl group. Then for $g, h \in \Aut(\calG)$, we have that the $W$-valued distance between the chambers $hx$ and  $ghx$ is $w \in W$ if and only if $h^{-1} gh \in Bw B$. In particular an automorphism $g$ is domestic if and only if $g$ is not conjugate to any element of $Bw_0B$, where $w_0$ denotes the long element of the associated Weyl group. The big Bruhat cell $Bw_0B$ is the only open cell in the Bruhat decomposition $G = \sqcup_{w \in W}Bw B$. From the algebraic point of view, this is the only top dimensional one. These considerations suggest that typical automorphisms are non-domestic. 
	\end{rem}
	
	For $g \in \Aut(\calG)$, we define
	\[\calO(g):=\{ x \in \calF \mid gx \text{ is opposite $x$}\}.\]
	We call $\calO(g)$ the \emph{opposite geometry} of $g$. The \emph{domestic locus} of $g$ is the set $\calD(g) := \calF - \calO(g)$.
	Of course, an automorphism $g$ is domestic if and only if $\mathcal D (g)=\mathcal F$.
	\begin{rem}
		The opposite geometry of an element $g$ often denotes all of the simplices (of any dimension) that are sent to opposite ones, but here we only consider chambers $\calO(g)$. 
	\end{rem}
	
	We will use a few facts about big cells for spherical buildings at infinity of affine buildings. These results hold in any dimension, but we stick to affine buildings of dimension 2 for consistency. 
	\begin{lem}\label{lem totally opposite open sets}
		Let $U,U'\subseteq\calF$ be open sets, and suppose that there exist
		opposite chambers $x\in U$ and $x'\in U'$. Then there exist open sets
		$V\subseteq U$ and $V'\subseteq U'$ containing $x$ and $x'$,
		respectively, such that every chamber in $V$ is opposite to every
		chamber in $V'$. In particular, for every chamber $x\in\chinf$, the
		big cell $\calO(x)$ is open.
	\end{lem}
	
	\begin{proof}
		By \cite[Lemma~6.5(iv)]{GrundhoferKramerVanMaldeghemWeiss2012}, the
		opposition relation is open in $\calF\times\calF$. Since $(x,x')$ is
		an opposite pair, there exist open neighborhoods $V\subseteq U$ of
		$x$ and $V'\subseteq U'$ of $x'$ such that every pair
		$(y,y')\in V\times V'$ consists of opposite chambers.
		
		The final assertion is \cite[Lemma~6.5(v)]{GrundhoferKramerVanMaldeghemWeiss2012}.
	\end{proof}
	
	\begin{lem}\label{lem open condition finite family}
		Let $F\subseteq\iso(X)$ be a finite set of isometries of the
		$2$-dimensional affine building $X$. Then the set
		\[
		\{x\in\calF\mid fx\text{ is opposite to }x
		\text{ for every }f\in F\}
		\]
		is open.
	\end{lem}
	
	\begin{proof}
        First note that if one of the elements $f\in F$ is domestic then the above set is empty (hence open). We shall assume therefore $F$ contains no domestic elements. 
        
		Fix $f\in F$, and choose $x\in\calF$ be such that $fx$ is opposite to
		$x$. By Lemma~\ref{lem totally opposite open sets}, there exist open
		neighborhoods $V$ of $x$ and $V'$ of $fx$ which are totally
		opposite. Since $f$ acts continuously on $\calF$, after shrinking
		$V$ if necessary we may assume that $fV\subseteq V'$. Hence $fy$ is
		opposite to $y$ for every $y\in V$. Therefore $\{x\in\calF\mid fx\text{ is opposite to }x\}$ is open.
		Taking the finite intersection over $f\in F$ proves the claim.
	\end{proof}
	
	We shall also need the following lemma.  
	
	\begin{lem}[{\cite[Corollary~2.5]{ciobotaru_le-bars26}}]\label{lem opp ch dense}
		For every $x \in \chinf$, the big cell $\calO(x)$ is dense in $\calF$.  
	\end{lem}

	\subsection{Loxodromic elements}
	
	Let $X$ be an affine building of dimension $2$, and let $\calG=(\calP,\calL,\calF)$ be its $n$-gon at infinity. For $\gamma\in\iso(X)$, its \emph{translation length} is $\ell(\gamma):=\inf_{x\in X}d(x,\gamma x)$, and its minimal set is $\min(\gamma):=\{x\in X\mid d(x,\gamma x)=\ell(\gamma)\}$. The isometry $\gamma$ is called \emph{hyperbolic} if $\ell(\gamma)>0$ and $\min(\gamma)\neq\emptyset$, or equivalently, if it admits an axis, that is an invariant geodesic on which the action is by translation. We call $\gamma$ \emph{loxodromic} if it is hyperbolic and $\min(\gamma)$ is a unique maximal flat.
	Loxodromic elements  exhibit the strongest possible form of contracting behavior on $X^\infty$.
	\begin{rem}
		If $X$ is the Bruhat-Tits building assoicated with a group $G$, then being a loxodromic element $g$ is equivalent to the Jordan projection $\lambda(g)$ belongs to the interior of the positive Weyl chamber. The same remark holds when $X = G/K$ is a symmetric space. Loxodromic elements are called ``strongly regular hyperbolic'' in \cite{caprace_ciobotaru15}.
	\end{rem}
	
	Another characterization of loxodromic elements is the following: the endpoints of any axis of $\gamma$ are in the interior of opposite chambers $x^+_g,$ (the attracting chamber) and $x^-_g $ (the repelling chamber) in $\calF$. 
	We summarize here some properties of loxodromic elements. 
	
	\begin{prop}\label{prop dyn srh}
		Let $g \in \iso(X)$ be a type preserving loxodromic element, with unique translation apartment at infinity $A$. Let $x^-, x^+ \in \chinf$ be the repelling (resp. attracting) chamber at infinity for $g$. Then for every $x \in \chinf$, the limit $\lim g^n (x) $ exists and coincides with $\rho_{A, x^-} (x)$, where $\rho_{A, x^-}$ is the retraction onto $A$ centered at $x^-$. The convergence is uniform on compact sets. 
	\end{prop}
	
	\begin{proof}
		This is extracted from \cite[Propositions~2.10 \& 2.11]{caprace_ciobotaru15} and their proofs. The fact that the convergence is uniform is well-known, and further described in \cite{kapovich_leeb_porti17anosov,kapovich_leeb_porti18morse} (where the authors describe more generally $\sigma_{mod}$-contracting sequences). 
	\end{proof}
	\begin{lem}\label{lem anosov schottky}
		Let $\gamma_1,\dots,\gamma_N\in\iso(X)$ be loxodromic elements.
		Assume that, for every $i\neq j$ and every
		$\varepsilon,\delta\in\{+,-\}$, the chambers $C_{\gamma_i}^{\varepsilon}$ and $C_{\gamma_j}^{\delta}$ are opposite. Then there exists $r\geq1$ such that $H:=
		\left\langle
		\gamma_1^r,\dots,\gamma_N^r
		\right\rangle$ is a free group of rank $N$, and every orbit map $H\longrightarrow X$ is a quasi-isometric embedding.
		
		Moreover, there exists a continuous, $H$-equivariant, antipodal	boundary map
		\[
		\psi\colon\partial H\longrightarrow\chinf.
		\]
		This map is dynamics-preserving: every non-trivial element
		$h\in H$ is loxodromic and, if
		$\eta_h^+,\eta_h^-\in\partial H$ are its attracting and repelling
		fixed points, respectively, then
		\[
		\psi(\eta_h^+)=C_h^+,
		\qquad
		\psi(\eta_h^-)=C_h^-.
		\]
	\end{lem}
	
	\begin{proof}
		For symmetric spaces, the Schottky assertion is the content of
		\cite[Theorem~3.47]{kapovich_leeb18}. The corresponding statement
		for affine buildings follows from the Morse lemma developed in
		\cite{kapovich_leeb_porti18morse}. The assertions concerning the
		boundary map and the dynamics of non-trivial elements are standard
		consequences of the Morse--Anosov property, see
		\cite{kapovich_leeb_porti17anosov} and
		\cite[Theorem~1.4]{kapovich_leeb_porti18morse}.
	\end{proof}
	
	\begin{cor}[Schottky products]\label{cor:schottky-products}
		Let $a,b\in\iso(X)$ be loxodromic elements such that
		\[
		C_a^\varepsilon
		\text{ is opposite }
		C_b^\delta
		\qquad
		\text{for every }
		\varepsilon,\delta\in\{+,-\}.
		\]
		Then there exists $r\geq1$ such that the elements
		\[
		h_n:=a^{rn}b^{rn},
		\qquad n\geq1,
		\]
		are loxodromic and satisfy
		\[
		C_{h_n}^+\longrightarrow C_a^+,
		\qquad
		C_{h_n}^-\longrightarrow C_b^-.
		\]
		
		In particular, if $U^+$ and $U^-$ are open neighborhoods of
		$C_a^+$ and $C_b^-$, respectively, then, for all sufficiently large
		$n$, one has
		\[
		C_{h_n}^+\in U^+,
		\qquad
		C_{h_n}^-\in U^-.
		\]
	\end{cor}
	
	\begin{proof}
		Let $r\geq1$ be given by Lemma~\ref{lem anosov schottky}, and set
		\[
		\alpha:=a^r,
		\qquad
		\beta:=b^r,
		\qquad
		H:=\langle\alpha,\beta\rangle.
		\]
		Then $H$ is free on the generators $\alpha$ and $\beta$. In
		particular,
		\[
		h_n=\alpha^n\beta^n
		\]
		is non-trivial for every $n\geq1$, and is therefore loxodromic by
		Lemma~\ref{lem anosov schottky}.
		
		Let $\eta_{h_n}^\pm\in\partial H$ be the attracting and repelling
		fixed points of $h_n$. With respect to the free basis
		$\{\alpha,\beta\}$, these points are represented by the infinite
		reduced words
		\[
		\eta_{h_n}^+
		=
		(\alpha^n\beta^n)^\infty
		\]
		and
		\[
		\eta_{h_n}^-
		=
		(\beta^{-n}\alpha^{-n})^\infty.
		\]
		Hence, in the boundary of $H$,
		\[
		\eta_{h_n}^+\longrightarrow\eta_\alpha^+,
		\qquad
		\eta_{h_n}^-\longrightarrow\eta_\beta^-.
		\]
		
		Let
		\[
		\psi\colon\partial H\longrightarrow\chinf
		\]
		be the boundary map from Lemma~\ref{lem anosov schottky}. By its
		continuity and the dynamics-preserving property,
		\[
		C_{h_n}^+
		=
		\psi(\eta_{h_n}^+)
		\longrightarrow
		\psi(\eta_\alpha^+)
		=
		C_\alpha^+
		=
		C_a^+,
		\]
		and similarly
		\[
		C_{h_n}^-
		=
		\psi(\eta_{h_n}^-)
		\longrightarrow
		\psi(\eta_\beta^-)
		=
		C_\beta^-
		=
		C_b^-.
		\]
		The final assertion follows from the openness of $U^+$ and $U^-$.
	\end{proof}

	\subsection{Limit set and groups of general type}\label{section:limit_set}
	Let $X$ be an affine building of dimension 2, and let $\calG = (\calP, \calL, \calF)$ be its $n$-gon at infinity.
	Let $\Gamma < \iso(X)$ be a group acting faithfully by isometries on $X$. In the sequel, we will deal with groups containing loxodromic elements and study their dynamics on the boundary, so it makes sense to consider the flag limit set  defined below. 
	\begin{defn}[Flag limit set]
		We define $\Lambda^{+}_\calF (\Gamma) \subseteq \calF$ as the set of attracting chambers of loxodromic elements in $\Gamma$. We define the flag limit set $\Lambda_\calF (\Gamma)$ as the closure of $\Lambda^{+}_\calF (\Gamma) \subseteq \chinf$ for the cone topology. 
	\end{defn}	
	We do not restrict ourselves to lattices of affine buildings. 
	\begin{defn}\label{def general type}
		We say that a subgroup $\Gamma < \iso(X)$ is \emph{of general type} if $\Gamma$ contains loxodromic elements and if for every $\Gamma$-invariant closed non-empty subset $M \subseteq \calF $ and for any finite set of chambers $x_1, \dots, x_N \in \chinf$, the intersection 
		\[M \cap \underset{i= 1, \dots,N}{\bigcap} \calO(x_i)\] is non-empty.
	\end{defn}
	
	The following shows that the flag limit set is a natural dynamical object.
	\begin{thm}\label{thm flag general type}
		Let $X$ be an affine building and let $\Gamma < \iso(X)$ be a group of general type. Then the limit set $\Lambda_\calF(\Gamma)$ is the smallest closed non-empty $\Gamma$-invariant subset of $\chinf$. It is $\Gamma$-minimal, perfect and uncountable. If $\Lambda_\calF(\Gamma) \neq \chinf$, then $\Lambda_\calF(\Gamma)$ has empty interior. If $\Gamma$ is a cocompact lattice, it has full limit set $\Lambda_\calF(\Gamma)= \chinf$. 
	\end{thm}
	\begin{proof}
		The last assertion is \cite[Proposition~6.1]{ciobotaru_le-bars26} and the rest of the statement is \cite[Theorem~4.4]{ciobotaru_le-bars26}. 
	\end{proof}

	For a Zariski--dense subgroup of a semisimple real algebraic group, Benoist proved that the attracting and repelling flags of proximal elements are dense in
	\(\Lambda(\Gamma)^2\) \cite[Lemma~3.6(iii)--(iv)]{benoist97}. The following proposition is an extension to the possibly exotic setting for groups of general type. 
	\begin{prop}\label{prop:prox_general_type}
		Let $\Gamma < \iso(X)$ be a group of general type acting properly on $X$.  Let $U^{+}, U^{-} \subseteq \chinf$ be non-empty open sets such that
		\[
		U^{+} \cap \Lambda_{\calF}(\Gamma) \neq \varnothing
		\qquad\text{and}\qquad
		U^{-} \cap \Lambda_{\calF}(\Gamma) \neq \varnothing.
		\]
		Assume furthermore that there exist opposite chambers $x \in U^{+} \cap \Lambda_{\calF}(\Gamma)$ and $y \in U^{-} \cap \Lambda_{\calF}(\Gamma)$. Then there exists a loxodromic element $h \in \Gamma$ with attracting chamber $C^{+}_{h} \in U^{+}$ and repelling chamber $C^{-}_{h} \in U^{-}$.
	\end{prop}

	\begin{proof}
		By Lemma~\ref{lem totally opposite open sets}, after replacing
		$U^+$ and $U^-$ by smaller open neighborhoods of $x$ and $y$,
		respectively, we may assume that $U^+$ and $U^-$ are totally
		opposite. By \cite[Proposition~6.4]{ciobotaru_le-bars26}, there exist a
		sequence $(s_n)_{n\geq1}$ in $\Gamma$ and a chamber
		$C_0\in\Lambda_{\calF}(\Gamma)$ such that $s_nC\longrightarrow C_0$ for every $C\in\chinf$. By Theorem~\ref{thm flag general type}, the action of $\Gamma$ on
		$\Lambda_{\calF}(\Gamma)$ is minimal. Since $U^+$ and $U^-$ both intersect $\Lambda_{\calF}(\Gamma)$, there exist
		$r_+,r_-\in\Gamma$ such that $r_+C_0\in U^+$ and $r_-C_0\in U^-$. Since $\Gamma$ is of general type, it contains a loxodromic element
		$\gamma$. Denote its attracting and repelling chambers by
		$C_\gamma^+$ and $C_\gamma^-$, respectively. For each $n\geq1$, set
		\[
		a_n:=(r_+s_n)\gamma(r_+s_n)^{-1},
		\qquad
		b_n:=(r_-s_n)\gamma(r_-s_n)^{-1}.
		\]
		Then $a_n$ and $b_n$ are loxodromic, with
		\[
		C_{a_n}^{\pm}=r_+s_nC_\gamma^\pm,
		\qquad
		C_{b_n}^{\pm}=r_-s_nC_\gamma^\pm.
		\]
		It follows that
		\[
		C_{a_n}^{\pm}\longrightarrow r_+C_0\in U^+,
		\qquad
		C_{b_n}^{\pm}\longrightarrow r_-C_0\in U^-.
		\]
		
		Choose $N\geq1$ sufficiently large that
		\[
		C_{a_N}^+,C_{a_N}^-\in U^+,
		\qquad
		C_{b_N}^+,C_{b_N}^-\in U^-,
		\]
		and set $a:=a_N$ and $b:=b_N$. Since $U^+$ and $U^-$ are totally opposite, $C_a^\varepsilon$ is opposite $C_b^\delta$ for every $\varepsilon,\delta\in\{+,-\}$. Corollary~\ref{cor:schottky-products} therefore yields a
		loxodromic element $h\in\langle a,b\rangle<\Gamma$ such that
		\[
		C_h^+\in U^+,
		\qquad
		C_h^-\in U^-.
		\]
	\end{proof}
	
	\subsection{Topological freeness on points and lines}
	Let $X$ be an affine building of dimension 2, and let $\calG = (\calP, \calL, \calF)$ be its $n$-gon at infinity. Let $\Gamma \curvearrowright X$ be a type-preserving action by isometries on $X$, and let $\calJ \in \{\calP, \calL, \calF\}$ be a type. Recall that a continuous group action is said to be topologically-free if the fixed-point set of any non-trivial element has empty interior. 
	\begin{lem}[Chamber topological freeness implies vertex topological freeness]
		\label{lem:top_free_chambers_implies_vertices}
		Let $\calG=(\calP,\calL,\calF)$ be a compact topological generalized polygon. Assume that every point row and every pencil of lines has no isolated points. Let \(g\in\Aut(\calG)\).
		If either the fixed-point set $\calP^g$ or the fixed-line set $\calL^g$ have a  non-empty interior, then the fixed-chamber set $\calF^g$ has non-empty interior.

		Consequently, if an action \(\Gamma\curvearrowright\calF\) is
		topologically free, then the induced actions on \(\calP\) and on
		\(\calL\) are topologically free.
	\end{lem}
	
	\begin{proof}
		We prove the statement about points, the statement for lines is dual.
		
		Suppose that \(\calP^g\) has non-empty interior. Let $U\subseteq \calP^g $ be a non-empty open set. We claim that every chamber in $\pi_{\calP}^{-1}(U)$ is fixed by \(g\). Since \(\pi_{\calP}:\calF\to\calP\) is continuous,
		the set \(\pi_{\calP}^{-1}(U)\) is a non-empty open subset of
		\(\calF\). This will prove the claim.
		
		Let
		\[
		C=(p,L)\in\pi_{\calP}^{-1}(U).
		\]
		Thus \(p\in U\), so \(gp=p\). We must prove that \(gL=L\).
		
		Since \(U\) is open in \(\calP\), the intersection $U\cap L^{\bI}$ is open in the point row \(L^{\bI}\). It is non-empty because it	contains \(p\). By assumption, \(L^{\bI}\) has no isolated points, so there exists $q\in U\cap L^{\bI}$ with \(q\neq p\). Since \(q\in U\subseteq\calP^g\), we have $gq= q$. Therefore the line \(gL\) contains both \(gp=p\) and \(gq=q\). But in a generalized polygon, two distinct collinear points are incident with a unique line. Hence $gL=L$ and \(C\in\calF^g\).
		
		We have shown that
		\[
		\pi_{\calP}^{-1}(U)\subseteq\calF^g.
		\]
		Since \(\pi_{\calP}^{-1}(U)\) is non-empty open, \(\calF^g\) has
		non-empty interior.
		
		The proof for lines is completely dual, and the final assertion follows immediately by applying the contrapositive	to every \(g\in\Gamma\setminus\{e\}\).
	\end{proof}

	Finally, we shall use the following result
	\begin{prop}[{\cite[Proposition~6.8]{ciobotaru_le-bars26}}]\label{prop top free general type}
		Let $\Gamma < \iso (X)$ be a subgroup of general type and assume that the action $ \Gamma \curvearrowright X$ is proper. Then the $\Gamma$-action on $\Lambda_\calF(\Gamma)$ is topologically free.
	\end{prop}

	We close this section by extending Definition~\ref{def:domestic} to the level of groups. 
	
	\begin{defn}\label{def:strong_non-dom}
		Let $\Gamma$ be a group and let $\Gamma \curvearrowright\calG $ be a continuous action by homeomorphisms on a compact generalized polygon $\calG =(\calP,\calL,\calF) $. 
		We say that the $\Gamma$-action on $\calF$ is \emph{strongly non-domestic} if for any finite set $F \subseteq \Gamma \setminus \{e\}$, the opposite geometry
		\[
		\Opp_\calF(F):=\bigcap_{f\in F}\calO_\calF(f)
		\]
		intersects $\Lambda_\calF(\Gamma)$.
	\end{defn}

	\section{Strong non-domesticity in projective planes}
	\label{sec:strong-nondomesticity-projective-planes}
	A generalized $3$-gon is called a \emph{projective plane}. Equivalently, a projective plane is a two-dimensional incidence geometry satisfying:
	\begin{enumerate}
		\item any two distinct points $p,q$ are incident to a unique line, denoted by $p\vee q$;
		\item any two distinct lines $L,L'$ are incident to a unique point, denoted by $L\wedge L'$;
		\item there exist four points no three of which are collinear.
	\end{enumerate}
	
	The purpose of this section is twofold. We first prove an intrinsic result about automorphisms of projective-planes: an element whose set of fixed flags is empty interior has a dense opposite geometry. We then deduce strong non-domesticity for groups of general type with full flag limit set.
	
	Let $\calG$ be a topological projective plane. Recall that $(V,\bI)$ denotes the corresponding incidence graph, endowed with its path metric $d$. Formally, $d$ is defined on points and lines (that is vertices of $(V,\bI)$), but we also use $d$ for the induced chamber distance (edges of  $(V,\bI)$).
	
	We will need the following basic topological lemma. 
	\begin{lem}\label{lem:topology}
		Let $f : Y \to Z$ be a continuous, open, and surjective map between topological spaces, and let $A \subseteq Y$. Let $Z_{0} \subseteq Z$ be a subset with empty interior, and assume that for every $z \in Z \setminus Z_{0}$ the intersection $A \cap f^{-1}(z)$ has empty interior in the fibre $f^{-1}(z)$, endowed with the subspace topology. Then $A$ has empty interior in $Y$.
	\end{lem}
	\begin{proof}
		Suppose, towards a contradiction, that $A$ contains a non-empty open subset $U\subseteq Y$. Since $f$ is open and surjective, $f(U)$ is a non-empty open subset of $Z$. Since $Z_0$ has empty interior, $f(U)$ is not contained in $Z_0$. Choose $z\in f(U)\setminus Z_0$. Then $U\cap f^{-1}(z)$ is a non-empty open subset of the fibre $f^{-1}(z)$, and it is contained in $A\cap f^{-1}(z)$, contradicting the hypothesis.
	\end{proof}
	Now we show that the set of fixed flags of a hyperbolic element of $X$ has empty interior. 
	\begin{prop}
		\label{prop:A2_hyperbolic_no_open_fixed_flags}
		Let \(X\) be a thick locally finite affine building of type \(\widetilde A_2\), and let \(\calG=(\calP,\calL,\calF)\) be its projective plane at infinity.
		Let \(g\in\iso(X)\) be type-preserving and hyperbolic. Then $\Fix_{\calF}(g) $ has empty interior. 
	\end{prop}
	
	\begin{proof}
		Since \(X\) is locally finite, it is proper as a \(\cat(0)\) space. For a
		hyperbolic isometry of a proper \(\cat(0)\) space, one has
		\[
		\Fix_{\partial X}(g)=\partial\operatorname{Min}(g),
		\]
		and
		\[
		\operatorname{Min}(g)\cong Y\times\mathbb R,
		\]
		where \(g\) acts by translation on the \(\mathbb R\)-factor, see for instance \cite[II.6]{bridson_haefliger99}.
		We distinguish two cases.
		
		First suppose that the endpoints of an axis of \(g\) lie in the interiors of opposite chambers of \(X^\infty\). Then \(g\) is loxodromic and \(\operatorname{Min}(g)\) is its unique translation apartment. Therefore \(\Fix_{\partial X}(g)\) is the boundary of this
		apartment. In particular, \(\Fix_{\calF}(g)\) is contained in the finite
		set of chambers of this apartment, and so has empty interior in \(\calF\).
		
		It remains to consider the singular case. Since \(X\) has type
		\(\widetilde A_2\), the endpoints of an axis of \(g\) are opposite
		vertices of the projective plane at infinity. Thus we may write them as a point \(a\in\calP\) and a line
		\(B\in\calL\), with \(a\) not incident to \(B\).
		
		Every fixed vertex at infinity lies in
		\(\partial\operatorname{Min}(g)\), hence lies on a Tits geodesic of
		length \(\pi\) joining \(a\) to \(B\). In a projective plane, such a
		geodesic has the form
		\[
		a \bI M \bI q \bI B,
		\]
		where \(M\) is a line through \(a\) and \(q\) is the unique point
		\(M\cap B\). Consequently,
		\[
		\calP^g\subseteq \{a\}\cup B^{\bI},
		\qquad
		\calL^g\subseteq a^{\bI}\cup\{B\}.
		\]
		The point row \(B^{\bI}\) and the pencil \(a^{\bI}\) are proper Schubert
		subsets of the compact topological projective plane, and hence have empty
		interior. Thus \(\calP^g\) has empty interior in \(\calP\), and
		\(\calL^g\) has empty interior in \(\calL\).
		
		Now suppose, for contradiction, that \(\Fix_{\calF}(g)\) has non-empty
		interior. Let \(U\subseteq\Fix_{\calF}(g)\) be a non-empty open set.
		Since the projections $\pi_{\calP}:\calF\to\calP $ and $ \pi_{\calL}:\calF\to\calL$ are open, \(\pi_{\calP}(U)\) and \(\pi_{\calL}(U)\) are non-empty open
		sets. But every chamber in \(U\) is fixed by \(g\), so $ \pi_{\calP}(U)\subseteq\calP^g$ and $\pi_{\calL}(U)\subseteq\calL^g$. This contradicts the fact that \(\calP^g\) and \(\calL^g\) have empty
		interior. Hence \(\Fix_{\calF}(g)\) has empty interior.
	\end{proof}

	\begin{lem}\label{lem:generic-opposite}
		Let $\calG$ be a topological projective plane in which all pencils and point rows have no isolated points. Let $g\in \Aut(\calG)$ be such that the fixed sets $\calJ^g$ have empty interior for $\calJ\in\{\calP,\calL,\calF\}$. Then the domestic locus $\calD(g)$ is closed with empty interior.
	\end{lem}
	\begin{proof}
		Since $\calD(g)=\calF\setminus\calO_\calF(g)$ and $\calO_\calF(g)$ is open by Lemma~\ref{lem open condition finite family}, the set $\calD(g)$ is closed. We prove that it has empty interior by applying Lemma~\ref{lem:topology} to the projection $\pi_\calP:\calF\to\calP$ and to $\calD(g)\subseteq\calF$. By Lemma~\ref{lem topo prop n gons}, the map $\pi_\calP$ is continuous, surjective and open.
		
		Let $p\in\calP\setminus\calP^g$. If $(p,L)\in\calD(g)$, then $(p,L)$ and $g(p,L)=(gp,gL)$ are not opposite. In a projective plane, two flags $(p,L)$ and $(q,M)$ fail to be opposite precisely when $p\bI M$ or $q\bI L$. Hence either $p\bI gL$, equivalently $g^{-1}p\bI L$, or $gp\bI L$. Since $p\neq gp$, this means that $L$ is one of the two lines $p\vee gp$ and $p\vee g^{-1}p$. Thus $\calD(g)\cap\pi_\calP^{-1}(p)$ is finite. Since the pencil $\pi_\calP^{-1}(p)$ has no isolated points, this finite subset has empty interior in the fibre.
		
		The exceptional set $\calP^g$ has empty interior by assumption. Lemma~\ref{lem:topology} therefore implies that $\calD(g)$ has empty interior.
	\end{proof}
	
	In particular, hyperbolic isometries of affine buildings of type $\tilde A_2$ have opposite dense geometry. We note that this is false in rank $\geq 3$:
	\begin{example} \label{ex:non-domestic}
		Let $\P$ be the projective space associated with the group $\mathrm {PGL}_d$, and consider the diagonal element $g=\mathrm {diag}(1,\dots,1,2)$. In this higher-dimensional setting flags are sequences $V_0\leq \dots \leq V_d$, where $V_i$ is a linear subspace of $\K^d$ of dimension $i$. Saying that such a flag is opposite to another flag $W_0\leq \dots \leq W_d$ amounts to saying that $V_i \cap W_{d-i}$ is trivial for each $i$.
		
		Assume that we have found a flag $V_0\leq \dots \leq V_d$ such that its image by $g$ is opposite to it. Thus $V_i\cap gV_{d-i}=0$, or equivalently $V_i+gV_{d-i}=\K^d$, for all $i$. Let $p$ denote the projection onto the first $d-1$ coordinates. It satisfies $pg=gp=p$. Hence $p(\K^d)=p(V_i+gV_{d-i})=p(V_i+V_{d-i})$.
		Choose $i$ with $i\leq d-i$. Since the flag is nested, $V_i\subseteq V_{d-i}$, and therefore $d-1=\dim p(\K^d)=\dim p(V_{d-i})\leq d-i$. Taking $i=2$ gives $d-1\leq d-2$ for every $d\geq4$, a contradiction. In other words, for $d\geq4$ such an opposite configuration cannot occur.
	\end{example}

	\begin{thm}[Open dense opposite geometry in projective planes]\label{thm:A2-open-dense-opposite-geometry}
		Let $\calG=(\calP,\calL,\calF)$ be a topological projective plane in which all pencils and point rows have no isolated points. Let $\Gamma\curvearrowright\calG$ be a type-preserving action by automorphisms. Assume that the action on $\calF$ is topologically free. Then, for every $\gamma\in\Gamma\setminus\{e\}$, the opposite geometry $\calO_\calF(\gamma)$ is open and dense in $\calF$.
	\end{thm}
	\begin{proof}
		Fix $\gamma\in\Gamma\setminus\{e\}$. By topological freeness on chambers, $\calF^\gamma$ has empty interior. By Lemma~\ref{lem:top_free_chambers_implies_vertices}, the fixed sets $\calP^\gamma$ and $\calL^\gamma$ have empty interior as well. Lemma~\ref{lem:generic-opposite} then gives that $\calD(\gamma)$ is closed with empty interior. Therefore $\calO_\calF(\gamma)=\calF\setminus\calD(\gamma)$ is open and dense.
	\end{proof}
	
	\begin{thm}[Strong non-domesticity in type $\tilde A_2$]\label{thm A2 strong domestic}
		Let $X$ be a locally finite thick affine building of type $\tilde A_2$, and let $\Gamma<\iso(X)$ be a subgroup of general type acting properly and type-preservingly. Assume that $\Lambda_\calF(\Gamma)=\chinf$. Then the action $\Gamma\curvearrowright\calF$ is strongly non-domestic.
	\end{thm}
	\begin{proof}
		Since $\Lambda_\calF(\Gamma)=\chinf$, Proposition~\ref{prop top free general type} gives topological freeness of the $\Gamma$-action on $\calF$. The projective plane at infinity has no isolated points in its rows and pencils, since it is the boundary at infinity of a locally finite thick affine building. Hence Theorem~\ref{thm:A2-open-dense-opposite-geometry} applies and shows that $\calO_\calF(f)$ is open and dense in $\calF$ for every $f\in\Gamma\setminus\{e\}$.
		
		Let $F\subseteq\Gamma\setminus\{e\}$ be finite. Then $\Opp_\calF(F)=\bigcap_{f\in F}\calO_\calF(f)$ is a finite intersection of open dense subsets of the compact Hausdorff space $\calF$. Hence $\Opp_\calF(F)$ is dense, and in particular non-empty. Since $\Lambda_\calF(\Gamma)=\calF$, this proves that $\Opp_\calF(F)$ intersects $\Lambda_\calF(\Gamma)$. Thus the action is strongly non-domestic.
	\end{proof}
	
	\section{Strong non-domesticity in generalized quadrangles}
	\label{sec:strong-nondomesticity-gq}
	The goal of this section is to prove strong non-domesticity for groups of general type acting on $\tilde C_2$-buildings with full flag limit set. Let $\calG=(\calP,\calL,\calF)$ be a compact generalized quadrangle, arising as the spherical building at infinity of a locally finite affine building of type $\tilde C_2$.  We keep the convention introduced above that collinearity and concurrency are reflexive: a point is collinear with itself and a line is concurrent with itself.  Hence two points are opposite if and only if they are not collinear, and two lines are opposite if and only if they are not concurrent.
	
	Observe that if $C=(p,L)$ and $D=(q,M)$ are chambers of a generalized quadrangle, then
	\begin{equation}\label{eq:gq-chamber-opposition}
		C\text{ is opposite }D
		\quad\Longleftrightarrow\quad
		p\text{ is opposite }q\text{ and }L\text{ is opposite }M.
	\end{equation}
	Equivalently, $C$ and $D$ are not opposite if and only if $p$ is collinear with $q$ or $L$ is concurrent with $M$. For a type-preserving automorphism $g\in\Aut(\calG)$, set $N_{\calP}(g):=\{p\in\calP\mid p\text{ is collinear with }gp\}$ and $N_{\calL}(g):=\{L\in\calL\mid L\text{ is concurrent with }gL\}$. These are the point and line vertices which are not sent to opposite vertices. Since the collinearity and concurrency relations are closed in compact topological polygons, $N_{\calP}(g)$ and $N_{\calL}(g)$ are closed.
	
	We shall also use the following consequences of Kramer's topology of compact polygons.
	\begin{lem}\label{lem:gq-no-isolated-schubert}
		Let $\calG=(\calP,\calL,\calF)$ be a non-discrete compact topological generalized quadrangle. Then point rows and pencils of lines have no isolated points. Moreover, for every point $p$ and every line $L$, the subsets \(p^{\bI}\), \(L^{\bI}\), \(p^\perp\), and \(L^\perp\) are closed with empty interior in the corresponding point or line space.
	\end{lem}
	\begin{proof}
		By \cite[Proposition~2.2.6]{kramer90}, discreteness of one row or one pencil in a compact topological polygon is equivalent to discreteness of the polygon. Since $\calG$ is non-discrete, no row or pencil is discrete, and in particular no row or pencil has an isolated point. The last assertion is a consequence of the fact that for any element $x \in \calJ$, with $\calJ \in \{\calP,\calL,\calF\}$, the set of opposite elements $\calO(x) $ is open \cite[Theorem~2.1.3]{kramer90} and dense in $\calJ$ by Lemma~\ref{lem opp ch dense}. 
	\end{proof}
	
	The next lemma is inspired by \cite[Lemma~3.5]{temmermans_thas_vanmaldeghem12}. Their argument is global, but the following version only uses an open family of non-opposite lines. Recall that for $g \in \Aut(\calG) $ and $\calJ \in \{\calP,\calL,\calF\}$, we denote by $\calJ^g$ the set of elements in $\calJ$ that are fixed by $g$. 
	
	\begin{lem}\label{lem:gq-local-fixed-incidence}
		Let $\calG$ be a non-discrete compact generalized quadrangle, and let $g\in\Aut(\calG)$ be type-preserving. Then
		\(\operatorname{Int}_{\calL} N_{\calL}(g)\subseteq (\calP^g)^{\bI}\). Dually, \(\operatorname{Int}_{\calP} N_{\calP}(g)\subseteq (\calL^g)^{\bI}\).
	\end{lem}
	\begin{proof}
		We prove the statement for lines, and the statement for points is obtained by duality. Let $U\subseteq N_{\calL}(g)$ be open and let $L\in U$. We shall show that $L$ is incident to a fixed point.
		
		Assume first that $L\neq gL$, and put \(x:=L\wedge gL\). Thus $x$ is the unique point incident with both $L$ and $gL$. We claim that $x$ is fixed. Suppose not, so that $gx\neq x$. The set $U\cap x^{\bI}$ is a non-empty open subset of the pencil $x^{\bI}$. By Lemma~\ref{lem:gq-no-isolated-schubert}, this pencil has no isolated points; in particular, $U\cap x^{\bI}$ is not contained in any prescribed finite subset. Choose \(M\in U\cap x^{\bI}\) with \(M\notin\{L,gL,g^{-1}L\}\).
		Since $M\in N_{\calL}(g)$, the lines $M$ and $gM$ are concurrent. We first note that $M\neq gM$. Indeed, if $M=gM$, then $M$ contains both $x$ and $gx$. As the line $gL$ also contains both $x$ and $gx$ and $x\neq gx$, we obtain by uniqueness that $x \vee gx = M=gL$, contradicting the choice of $M$.
		
		Let \(y:=M\wedge gM\). The incidence graph then contains the cycle \[x\bI M\bI y\bI gM \bI gx \bI gL \bI x.\]
		We check that it is non-degenerate. The three lines $M$, $gM$ and $gL$ are pairwise distinct. Indeed $M\neq gM$ as shown above, $M\neq gL$ by choice of $M$, and $gM\neq gL$ because otherwise $M=L$. Also $y\neq x$, for otherwise $gM$ would contain both $x$ and $gx$, hence $gM=gL$. Finally $y\neq gx$, for otherwise $M$ would contain both $x$ and $gx$, hence $M=gL$. Thus we have obtained a genuine $6$-cycle in the incidence graph, contradicting that a generalized quadrangle has girth $8$. Therefore $gx=x$, and $L$ contains a fixed point.
		
		Assume now that $L=gL$. We show that every point incident to $L$ is fixed. Let $x\in L$, and suppose $gx\neq x$. Choose \(M\in U\cap x^{\bI}\), with \(M\neq L\), which is possible by Lemma~\ref{lem:gq-no-isolated-schubert}. As above, $M\neq gM$. Indeed, otherwise both $M$ and $L$ would contain both $x$ and $gx$, forcing $M=L$. Since $M\in U\subseteq N_{\calL}(g)$, the lines $M$ and $gM$ are concurrent. Set $y=M\wedge gM$. Then \[x\bI M\bI y\bI gM \bI gx \bI gL \bI x\] is again a non-degenerate $6$-cycle. Indeed, the possible degeneracies would force either $M=L$ or $gM=L$ as before. Again, this leads to a contradiction with the girth, and $gx=x$. Therefore every point of $L$ is fixed, and in particular $L$ contains a fixed point.
		
		We have shown that every line in $U$ belongs to $(\calP^g)^{\bI}$. Taking $U=\operatorname{Int}_{\calL} N_{\calL}(g)$ gives the desired inclusion.
	\end{proof}
	
	For opposite points $a,b\in\calP$, define their common-neighbour set \(Q(a,b):=a^\perp\cap b^\perp \subseteq \calP\).
	The following lemma replaces the finite-fibre argument used in projective planes.
	
	\begin{lem}\label{lem:gq-common-neighbour-incidence}
		Let $\calG$ be a non-discrete compact topological generalized quadrangle and let $a,b\in\calP$ be opposite points. Then \(Q(a,b)^{\bI}\) is closed with empty interior in $\calL$. Dually, if $L, M\in\calL$ are opposite lines, then the set $(L^\perp \cap M^\perp)^{\bI}$ is closed with empty interior in $\calP$.
	\end{lem}
	\begin{proof}
		As before we only prove the statement for points. Consider the map
		\[
		\alpha:q \in a^\perp\setminus\{a\}\longmapsto a\vee q \in a^{\bI},
		\]
		where $a\vee q$ is the unique line through the distinct collinear points $a$ and $q$. By \cite[Lemma~2.1.6]{kramer90}, this is an open surjection whose fibre over a line $L\in a^{\bI}$ is $L^{\bI}\setminus\{a\}$.
		
		Fix $L \in a^{\bI}$. Since $a$ and $b$ are opposite, $b$ is not incident with $L$. The generalized quadrangle axiom gives a unique point $q_L\in L^{\bI}$ collinear with $b$, and $q_L\neq a$. Hence \(Q(a,b)\cap\alpha^{-1}(L)=\{q_L\}\).
		By Lemma~\ref{lem:gq-no-isolated-schubert}, the singleton $\{q_L\}$ has empty interior in the fibre $L^{\bI}\setminus\{a\}$. Lemma~\ref{lem:topology} therefore gives
		\begin{equation}\label{eq:Q-empty-interior}
			\operatorname{Int}_{a^\perp\setminus\{a\}}Q(a,b)=\emptyset.
		\end{equation}
		
		Next consider
		\[
		\beta: M \in \calL\setminus a^{\bI}\longmapsto \beta(M) a^\perp\setminus\{a\},
		\]
		where $\beta(M) = L^{\bI} \cap a^\perp$ is the unique point of $M$ collinear with $a$. This is the dual form of the $n=4$ bundle in \cite[Lemma~2.1.7]{kramer90}, and hence $\beta$ is open. For \(M\notin a^{\bI}\), uniqueness in the generalized quadrangle gives \(M\in Q(a,b)^{\bI}\) if and only if \(\beta(M)\in Q(a,b)\). Thus \(Q(a,b)^{\bI}\cap(\calL\setminus a^{\bI})=\beta^{-1}(Q(a,b))\).
		If the right-hand side contained a non-empty open set, its image under the open map $\beta$ would be a non-empty open subset of $Q(a,b)$, contradicting \eqref{eq:Q-empty-interior}. Therefore $Q(a,b)^{\bI}\cap(\calL\setminus a^{\bI})$ has empty interior in $\calL\setminus a^{\bI}$. Since the pencil $a^{\bI}$ is closed with empty interior by Lemma~\ref{lem:gq-no-isolated-schubert}, the set $Q(a,b)^{\bI}$ has empty interior in $\calL$.
		
		Finally, $Q(a,b)$ is closed, hence compact, the incidence relation is closed by \cite[Proposition~2.1.12]{kramer90}. Therefore \(\{(q,M)\in\calF\mid q\in Q(a,b)\}\) is compact, and its projection to $\calL$ is precisely $Q(a,b)^{\bI}$. Since $\calL$ is Hausdorff, $Q(a,b)^{\bI}$ is closed.
	\end{proof}
	
	We now aim at controlling the incidence images of the fixed vertices of a hyperbolic isometry in an affine building. 
	
	\begin{prop}\label{prop:gq-fixed-incidence-hyperbolic}
		Let $X$ be a thick locally finite affine building of type $\tilde C_2$. Let $g\in\iso(X)$ be type-preserving and hyperbolic. Then \((\calP^g)^{\bI}\) and \((\calL^g)^{\bI}\) are closed with empty interior.
	\end{prop}

	\begin{proof}
		The incidence images are closed because $\calP^g$ and $\calL^g$ are closed and the incidence relation is closed. It remains to prove that they have empty interior.
		
		Since $X$ is locally finite, it is proper as a $\cat(0)$ space. As in the proof of Proposition~\ref{prop:A2_hyperbolic_no_open_fixed_flags}, one has $\Fix_{\bd X}(g)=\bd\operatorname{Min}(g)$ and the minimal set splits canonically as $\operatorname{Min}(g)\cong Y\times\mathbb R$ where $g$ acts by translation on the $\mathbb R$-factor. Hence $\bd\operatorname{Min}(g)=\bd Y*\{g^-,g^+\} $, where $g^-,g^+\in X^\infty$ are the endpoints of an axis of $g$ and we denoted by $*$ the spherical join.
		
		First suppose that $g^-$ and $g^+$ lie in the interiors of opposite chambers. Then $g$ is loxodromic by \cite[Lemma~2.3]{caprace_ciobotaru15} and $\operatorname{Min}(g)$ is its unique translation apartment. Thus the fixed boundary is the boundary apartment determined by $g^-$ and $g^+$. It contains only finitely many point and line vertices. Therefore $\calP^g$ and $\calL^g$ are finite, and their incidence images are finite unions of point rows and pencils. By Lemma~\ref{lem:gq-no-isolated-schubert}, these have empty interior.
		
		It remains to consider the singular case. Since the spherical building has rank $2$, the endpoints $g^-$ and $g^+$ are opposite vertices. Suppose first that they are point vertices, and write them as $a^-$ and $a^+$ in $\calP$. Every fixed vertex at infinity lies in $\bd\operatorname{Min}(g)$, hence on a geodesic (for the Tits distance) of length $\pi$ joining $a^-$ to $a^+$. In a spherical apartment of type $C_2$, such a geodesic has the form 
		\[a^- \bI L^- \bI q  \bI L^+ \bI a^+.\]
		Consequently, \(\calP^g\subseteq \{a^-,a^+\}\cup Q(a^-,a^+)\) and \(\calL^g\subseteq (a^-)^{\bI}\cup (a^+)^{\bI}\).
		Therefore
		\[
		(\calP^g)^{\bI}
		\subseteq
		(a^-)^{\bI}\cup(a^+)^{\bI}\cup
		Q(a^-,a^+)^{\bI},
		\]
		which has empty interior by Lemmas~\ref{lem:gq-no-isolated-schubert} and \ref{lem:gq-common-neighbour-incidence}. Also \((\calL^g)^{\bI}\subseteq (a^-)^\perp\cup(a^+)^\perp\), which has empty interior by Lemma~\ref{lem:gq-no-isolated-schubert}. The case in which $g^-$ and $g^+$ are lines is dual.
	\end{proof}
	
	\begin{thm}[Open dense opposite geometry in type $\tilde C_2$]\label{thm:gq-open-dense-opposite-geometry}
		Let $X$ be a thick locally finite affine building of type $\tilde C_2$. Let $g\in\iso(X)$ be a type-preserving hyperbolic isometry. Then the opposite geometry \(\calO_{\calF}(g)=\{C\in\calF\mid gC\text{ is opposite }C\}\) is open and dense in $\calF$.
	\end{thm}
	\begin{proof}
		Openness was proved in Lemma~\ref{lem open condition finite family}. We prove density, or equivalently that the domestic locus $\calD(g)=\calF\setminus\calO_{\calF}(g)$ has empty interior.
		
		By Lemma~\ref{lem:gq-local-fixed-incidence} and Proposition~\ref{prop:gq-fixed-incidence-hyperbolic},
		\[
		\operatorname{Int}_{\calL}N_{\calL}(g)=\emptyset,
		\qquad
		\operatorname{Int}_{\calP}N_{\calP}(g)=\emptyset.
		\]
		A chamber $C=(p,L)$ belongs to $\calD(g)$ if and only if $p\in N_{\calP}(g)$ or $L\in N_{\calL}(g)$, thus
		\[
		\calD(g)=
		\pi_{\calP}^{-1}(N_{\calP}(g))
		\cup
		\pi_{\calL}^{-1}(N_{\calL}(g)).
		\]
		The projections $\pi_{\calP}$ and $\pi_{\calL}$ are continuous open surjections by Lemma~\ref{lem topo prop n gons}. Since $N_{\calP}(g)$ and $N_{\calL}(g)$ are closed with empty interior, their preimages under these open maps are closed with empty interior. Hence $\calD(g)$ is a finite union of closed sets with empty interior. Therefore $\calO_{\calF}(g)$ is dense.
	\end{proof}
	
	\begin{cor}[Strong non-domesticity in type $\tilde C_2$]\label{cor:gq-strong-nondomesticity}
		Let $X$ be a thick locally finite affine building of type $\tilde C_2$, and let $\Gamma<\iso(X)$ act properly and type-preservingly. Assume that:
		\begin{enumerate}
			\item $\Gamma$ is of general type;
			\item $\Lambda_{\calF}(\Gamma)=\calF$;
			\item every non-trivial element of $\Gamma$ is hyperbolic.
		\end{enumerate}
		Then $\Gamma$ is strongly non-domestic. In particular, the conclusion holds for type-preserving cocompact lattices acting freely on $X$.
	\end{cor}
	\begin{proof}
		Let $f\in\Gamma\setminus\{e\}$. By assumption, $f$ is hyperbolic, so Theorem~\ref{thm:gq-open-dense-opposite-geometry} gives that $\calO_{\calF}(f)$ is open and dense in $\calF$.
		
		If \(F\subseteq\Gamma\setminus\{e\}\) is finite, then \(\Opp_{\calF}(F)=\bigcap_{f\in F}\calO_{\calF}(f)\) is a finite intersection of open dense subsets of the compact Hausdorff space $\calF$, hence is dense and in particular non-empty. 
		
		Any isometry of an affine building is semisimple \cite{parreau00}. For a free cocompact lattice, freeness rules out non-trivial elliptic elements, hence every non-trivial lattice element is hyperbolic. The remaining hypotheses, namely being of general type with full limit set, are satisfied for cocompact lattices by \cite[Proposition~6.1]{ciobotaru_le-bars26}.
	\end{proof}

	\section{Transversal contractivity from strong non-domesticity}
	\label{sec:P-star-from-strong-nondomesticity}
	
	Let \(X\) be an irreducible affine building of type $\tilde A_2, \tilde C_2$ or $\tilde{G}_2$, let \(\calG=(\calP,\calL,\calF)\) be its topological generalized polygon at infinity, and let $\Gamma \curvearrowright X$ be an isometric action of general type. In this section we convert strong non-domesticity into transversal contractivity. We will show that $\Gamma \curvearrowright \mathcal S$ is transversely contractive, with $\mathcal S$ a compact subset of $\calG$. Crucially, the construction depends on the type of the building. Given a chamber $C \in \calF$, we denote by \(\xi_C \in \bd X\) the Tits barycentre of \(C\), and $\overline B_T(\xi_C, \theta) $ is the closed ball of radius $\theta $ around $\xi_C$. We introduce the following notations:
	\begin{itemize}
		\item If $X$ is of type $\tilde A_2$: for $C \in \calF$, set 
		\[
		\mathcal B(C):=\{D\in\calF\mid d_{\calF}(C,D)\leq 1\},
		\]
		where \(d_{\calF}\) denotes distance in the chamber graph.
		\item If $X$ is of type $\tilde C_2$ or $\tilde G_2$: for $C \in \calF$, set 
		\[
		\mathcal B(C)=\overline B_T(\xi_C,\pi/2)\cap(\calP\sqcup\calL).
		\]
		where \(d_{\calF}\) denotes distance in the chamber graph.
	\end{itemize}
	
	Equivalently, in type $\tilde{A}_2$, 
	\[
	\mathcal B(C):=\{D\in\calF\mid d_{\calF}(C,D)\leq 1\},
	\]
	where \(d_{\calF}\) denotes distance in the chamber graph. In type $\tilde C_2$ or $\tilde G_2$, if the generalized polygon has gonality \(n=4\) or \(n=6\), we have
	\[
	\mathcal B(C):=\{v\in\calP\sqcup\calL\mid d(v,C)\leq n/2 - 1\},
	\]
	where \(d(v,C):=\min\{d(v,p_C),d(v,L_C)\}\) is computed in the incidence
	graph, see Figures~\ref{figure B(C) A2} and \ref{figure B(C) C2-G2}.
	
	\begin{figure}[htbp]
		\centering
		\begin{minipage}{.5\textwidth}
			\centering
			\begin{tikzpicture}[scale=2]
				\draw [thick, draw=black, fill=yellow, opacity=0.2]
				(0.5, 0.86) -- (0,0) --  (1,0)  ;	
				\draw [thick, draw=black, fill=yellow, opacity=0.2]
				(1,0) -- (0,0) -- (0.5,-0.86) ;	
				\draw [thick, draw=black, fill=yellow, opacity=0.2]
				(0.5,-0.86) -- (0,0) --  (-0.5,-0.86)  ;	
				\draw (-1,0) -- (1,0)  ;
				\draw (-0.5, -0.86) -- (0.5, 0.86)  ;
				\draw[black] (-0.5, 0.86) -- (0.5, -0.86) ;
				\filldraw (0.5, 0.86) circle(1pt) ;
				\filldraw (-0.5, -0.86) circle(1pt) ;
				\draw (1,-0.4) node{$C$} ;
			\end{tikzpicture}
			\caption{Type $\tilde{A}_2$: the shaded region represents the intersection of $\calB(C)$ with an apartment containing $C$.}\label{figure B(C) A2}
		\end{minipage}%
		\begin{minipage}{.5\textwidth}
			\centering
			\begin{tikzpicture}[scale=2]
				\draw [dashed] (0.38,0.92) -- (-0.38,-0.92);
				\draw (-1,0) -- (1,0)  ;
				\draw (0,-1) -- (0,1)  ;
				\draw (-0.71, -0.71) -- (0.71, 0.71);
				\draw (-0.71, 0.71) -- (0.71, -0.71);
				\draw (0.71, -0.71) circle(1pt) ;
				\draw (0.71, 0.71) circle(1pt) ;
				\draw (1,0) circle(1pt) ;
				\draw (0, -1) circle(1pt) ;
				\draw (1,-0.4) node{$C$} ;
			\end{tikzpicture}
			\caption{Type $\tilde{C}_2$: the vertices with circles represent the intersection of $\calB(C)$ with an apartment containing $C$.}\label{figure B(C) C2-G2}
		\end{minipage}%
	\end{figure}

	\begin{lem}\label{lem:rank-two-block-separation}
		In the three cases above, for chambers \(C,D\in\calF\), one has
		\[
		\mathcal B(C)\cap\mathcal B(D)=\emptyset
		\quad\Longleftrightarrow\quad
		C\text{ is opposite }D.
		\]
		Moreover, \(\mathcal B(C)\) is closed in \(S\), and
		\(g\mathcal B(C)=\mathcal B(gC)\) for every type-preserving
		automorphism \(g\).
	\end{lem}
	
	\begin{proof}
		Equivariance is immediate from the definitions.
		
		In type \(\tilde A_2\), opposite chambers are precisely the chambers at distance \(3\) in the chamber graph. Thus two radius-one chamber blocks meet if and only if the two chambers are at distance at most \(2\), which is equivalent to being non-opposite.
		
		In types \(\tilde C_2\) and \(\tilde G_2\), it suffices to check the statement in an apartment containing both chambers. There, the two panel half-balls \(\overline B_T(\xi_C,\pi/2)\cap(\calP\sqcup\calL)\) and \(\overline B_T(\xi_D,\pi/2)\cap(\calP\sqcup\calL)\) are disjoint exactly when \(C\) and \(D\) are opposite.
		
		Write $C= (p, L)$. The set $\calB(C)$ is a finite union of Schubert varieties in the sense of \cite[\S1.7]{kramer90}. Indeed, in type $A_2$, \(\mathcal B(C)= \calF_1(p,L) \cup \calF_1(L,p)\), with the notations of \cite{kramer90}. In type $C_2$, \(\mathcal B(C)= \mathrm{Cl}\calP_1(L,p) \cup \mathrm{Cl}\calL_1(p,L)\) and in type $G_2$,  \(\mathcal B(C)= \mathrm{Cl}\calP_2(p,L) \cup \mathrm{Cl}\calL_2(L,p)\), again with the notations of \cite{kramer90}. Schubert varieties are closed by \cite[Proposition~2.1.11]{kramer90}, thus $\calB(C)$ is also closed. 
	\end{proof}
	
	From now on, write $\calS = \calF$ if $X$ is of type $\tilde A_2$ and $\calS = \calP \sqcup \calL$ if $X$ is of type $\tilde C_2$ or $\tilde G_2$.
	\begin{prop}\label{prop:rank-two-block-contraction}
		Let \(h\) be loxodromic, with attracting and repelling chambers
		\(C_h^+\) and \(C_h^-\). Let \(V^+\) and \(V^-\) be open neighbourhoods
		of \(\mathcal B(C_h^+)\) and \(\mathcal B(C_h^-)\) in \(\calS\). Then there
		exists \(m_0\geq1\) such that
		\[
		h^m(S\setminus V^-)\subseteq V^+
		\qquad(m\geq m_0).
		\]
	\end{prop}
	
	\begin{proof}
		Set \(K:=S\setminus V^-\). Then \(K\) is compact and is disjoint from
		\(\mathcal B(C_h^-)\).
		
		Let \(A\) be the translation apartment of \(h\), and let \(\rho=\rho_{A,C_h^-}\). The retraction \(\rho\) preserves Weyl position with respect to \(C_h^-\). Therefore, by checking in the corresponding rank-two Coxeter complex, every point of \(\calS\setminus\mathcal B(C_h^-)\)	retracts into \(\mathcal B(C_h^+)\). Hence
		\(\rho(K)\subseteq \mathcal B(C_h^+)\subseteq V^+\). The conclusion follows from Proposition~\ref{prop dyn srh}.
	\end{proof}

	\begin{lem}
		\label{lem:prepare-endpoints-strong-nondomestic}
		Let \(\Gamma<\iso(X)\) act properly and type-preservingly. Assume that
		\(\Gamma\) is of general type and that the action
		\(\Gamma\curvearrowright\calF\) is strongly non-domestic. Let
		\(F\subseteq\Gamma\setminus\{e\}\) be finite. Then there exist opposite
		chambers \(C^+,C^-\in\Lambda_{\calF}(\Gamma)\) such that, for every
		\(f\in F\) and every \(\sigma,\tau\in\{+,-\}\), the chambers
		\(fC^\sigma\) and \(C^\tau\) are opposite.
	\end{lem}
	
	\begin{proof}
		Since \(\Gamma\) is of general type, it contains a loxodromic element and
		is therefore infinite. Choose
		\(a\in\Gamma\setminus(\{e\}\cup F\cup F^{-1})\). Set
		\[
		L:=
		F
		\cup a^{-1}Fa
		\cup a^{-1}F
		\cup Fa
		\cup\{a\}.
		\]
		By the choice of \(a\), every element of \(L\) is non-trivial.
		
		By strong non-domesticity, choose \(C^+\in\Lambda_{\calF}(\Gamma)\) such that \(\ell C^+\) is opposite to \(C^+\) for every \(\ell\in L\). Put \(C^-:=aC^+\). Since
		\(a\in L\), the chambers \(C^+\) and \(C^-\) are opposite.
		
		Let \(f\in F\). Then \(fC^+\) opposite \(C^+\) because \(f\in L\). The chambers \(fC^-\) and \(C^-\) are opposite because \(a^{-1}faC^+\) is opposite to \(C^+\), and the chambers \(fC^+\) and \(C^-\) are opposite because \(a^{-1}fC^+\) is opposite to \(C^+\). Finally, \(fC^-\) opposite \(C^+\) follows from \(fa\in L\). This proves the claim.
	\end{proof}
	The following is Theorem~\ref{Thm:intro-str_non-dom_trans_cont} from the introduction. 
	\begin{thm}[Strong non-domesticity implies transversal contractivity]
		\label{thm:rank-two-strongndom-Pstar}
		Let \(X\) be a locally finite affine building of type
		\(\tilde A_2\), \(\tilde C_2\), or \(\tilde G_2\). Let
		\(\Gamma<\iso(X)\) act properly and type-preservingly. Assume that
		\(\Gamma\) is of general type and that the action
		\(\Gamma\curvearrowright\calF\) is strongly non-domestic. Then the
		action \(\Gamma\curvearrowright \calS\) is transversely contractive.
	\end{thm}
	
	\begin{proof}
		Let \(F\subseteq\Gamma\setminus\{e\}\) be finite.
		By Lemma~\ref{lem:prepare-endpoints-strong-nondomestic}, choose opposite
		chambers \(C^+,C^-\in\Lambda_{\calF}(\Gamma)\) such that, for all
		\(f\in F\) and all \(\sigma,\tau\in\{+,-\}\), the chambers
		\(fC^\sigma\) and \(C^\tau\) are opposite. Since opposition is open, we
		may choose open neighbourhoods \(U^+\) of \(C^+\) and \(U^-\) of \(C^-\)
		such that $fy^\sigma$ is opposite $z^\tau $ for all $f\in F,\ y^\sigma\in U^\sigma,\ z^\tau\in U^\tau$. 
		Here \(\sigma,\tau\in\{+,-\}\).
		
		We now use Proposition~\ref{prop:prox_general_type}. The open sets
		\(U^+\) and \(U^-\) meet \(\Lambda_{\calF}(\Gamma)\), and they contain
		the opposite limit chambers \(C^+\) and \(C^-\). Hence there exists a
		loxodromic element \(h\in\Gamma\) such that
		\(C_h^+\in U^+\) and \(C_h^-\in U^-\).
		
		By Lemma~\ref{lem:rank-two-block-separation}, the closed sets \(\calB(C_h^+)\) and
		\(\calB(C_h^-)\) are disjoint. Moreover, by the choice of \(U^+\) and \(U^-\),
		the chambers \(fC_h^\sigma\) and \(C_h^\tau\) are opposite for all
		\(f\in F\) and all signs \(\sigma,\tau\). Using equivariance of the
		blocks and Lemma~\ref{lem:rank-two-block-separation}, we get
		\(f\calB(C_h^\sigma)\cap \calB(C_h^\tau)=\emptyset\) for all such \(f,\sigma,\tau\).
		
		Since \(S\) is compact Hausdorff, hence normal, and since only finitely
		many closed-set separation conditions occur, we may choose open
		neighbourhoods \(V^\pm\supseteq \calB(C_h^{\pm})\) such that
		\(V^+\cap V^-=\emptyset\) and
		\(fV^\sigma\cap V^\tau=\emptyset\) for every \(f\in F\) and all signs
		\(\sigma,\tau\).
		
		Choose \(m\) large enough so that
		\(h^m(S\setminus V^-)\subseteq V^+\), which is possible by
		Proposition~\ref{prop:rank-two-block-contraction}. Put
		\(\gamma:=h^m\). Then \(\gamma(S\setminus V^-)\subseteq V^+\), the sets
		\(V^+\) and \(V^-\) are disjoint, and, writing \(V=V^+\cup V^-\), one
		has \(fV\cap V=\emptyset\) for every \(f\in F\). 
	\end{proof}
	The following easy lemma shows that being transversely contractive is an intrinsic property of a group. 
	\begin{lem}\label{lem:pullback-left-regular-transversal}
		If \(\Gamma\) admits a transversely contractive action on a non-empty set
		\(S\), then the left-regular action \(\Gamma\curvearrowright\Gamma\) is
		transversely contractive.
	\end{lem}
	
	\begin{proof}
		Choose \(s_0\in S\). Given elements \(\gamma_i\) and subsets
		\(U_i^\pm\subseteq S\) witnessing transversal contractivity for the
		action on \(S\), set
		\[
		C_i^\pm:=\{a\in\Gamma\mid as_0\in U_i^\pm\}.
		\]
		If \(a\notin C_i^-\), then \(as_0\notin U_i^-\), hence
		\(\gamma_i as_0\in U_i^+\), so \(\gamma_i a\in C_i^+\). Pairwise
		disjointness and the conditions \(fC\cap C=\emptyset\), where
		\(C=\bigcup_i(C_i^+\cup C_i^-)\), follow by applying the orbit map
		\(a\mapsto as_0\) to the corresponding conditions for the sets
		\(U_i^\pm\). Thus the left-regular action is transversely contractive.
	\end{proof}
	
	\begin{cor}\label{cor:A2-Pstar-full-limit}
		Let \(X\) be a locally finite affine building of type \(\tilde A_2\),
		and let \(\Gamma<\iso(X)\) act properly and type-preservingly. Assume
		that \(\Gamma\) is of general type and that
		\(\Lambda_{\calF}(\Gamma)=\calF\). Then
		\(\Gamma\curvearrowright\calF\) is transversely contractive. In
		particular, the conclusion holds for type-preserving cocompact lattices.
	\end{cor}
	
	\begin{proof}
		By Theorem~\ref{thm A2 strong domestic}, the action on \(\calF\) is
		strongly non-domestic. Theorem~\ref{thm:rank-two-strongndom-Pstar} gives
		the conclusion. For cocompact lattices, the hypotheses of general type
		and full flag limit set follow from
		\cite[Proposition~6.1]{ciobotaru_le-bars26}.
	\end{proof}
	
	\begin{cor}\label{cor:C2-Pstar-full-limit}
		Let \(X\) be a locally finite affine building of type \(\tilde C_2\),
		and let \(\Gamma<\iso(X)\) act properly and type-preservingly. Assume
		that \(\Gamma\) is of general type, that
		\(\Lambda_{\calF}(\Gamma)=\calF\), and that every non-trivial element of
		\(\Gamma\) is hyperbolic. Then
		\(\Gamma\curvearrowright\calP\sqcup\calL\) is transversely contractive,
		and consequently \(\Gamma\) is transversely contractive. In particular,
		the conclusion holds for free type-preserving cocompact lattices.
	\end{cor}
	
	\begin{proof}
		By Corollary~\ref{cor:gq-strong-nondomesticity}, the action on
		\(\calF\) is strongly non-domestic. Theorem~\ref{thm:rank-two-strongndom-Pstar}
		gives transversal contractivity for the action on the panel space. The
		pullback to the left-regular action follows from
		Lemma~\ref{lem:pullback-left-regular-transversal}.
		
		For free type-preserving cocompact lattices, the hypotheses of general
		type and full flag limit set follow from
		\cite[Proposition~6.1]{ciobotaru_le-bars26}, and freeness rules out
		non-trivial elliptic elements.
	\end{proof}
	
    \section{Applications}
    The purpose of this section is explain how to deduce MIF and selflessness from transversal contractivity, and thus to complete the proofs of Theorem \ref{Thm:MIF} and Theorem \ref{thm:intro-selflessness}
    \subsection{Mixed identities and selflessness}
	
	\begin{prop}\label{prop:trans-contr_mif}
		Let $\Gamma$ be a non-trivial transversely contractive group. Then $\Gamma$ is mixed-identity-free.
		
		More precisely, for every finite non-empty set $F\subset \Gamma\setminus\{e\}$,
		there exists $\gamma\in\Gamma$ such that, for every non-trivial word
		$w\in \Gamma * \langle x\rangle$ whose coefficients belong to $F$, one has
		$w(\gamma)\neq e$.
	\end{prop}
	
	\begin{proof}
		Let $F\subset \Gamma\setminus\{e\}$ be finite and non-empty. Let $\Gamma\curvearrowright S$ be a transversely contractive action. Applying this property  with $n=1$ and with the
		finite set $F$, we obtain an element $\gamma\in\Gamma$ and disjoint subsets
		$U^+,U^-\subset S$ such that, writing $U=U^+\cup U^-$, we have
		\[
		\gamma(S\setminus U^-)\subset U^+
		\qquad\text{and}\qquad
		fU\cap U=\varnothing \quad\text{for every } f\in F.
		\]
		The first inclusion also implies
		\[
		\gamma^{-1}(S\setminus U^+)\subset U^-.
		\]

		Since $F$ is non-empty, the condition $fU\cap U=\varnothing$ implies that
		$U\neq S$. Fix $p\in S\setminus U$. We note that, for every non-zero integer
		$m$,
		\[
		\gamma^m(S\setminus U)\subset U^+
		\quad\text{if } m>0,
		\qquad\text{and}\qquad
		\gamma^m(S\setminus U)\subset U^-
		\quad\text{if } m<0.
		\]
		This follows from the two inclusions above and from the disjointness of
		$U^+$ and $U^-$. In particular, $\gamma^m(S\setminus U)\subset U$ for every
		$m\neq 0$, and hence $\gamma$ has infinite order.
		
		Let now $w\in \Gamma * \langle x\rangle$ be non-trivial, and assume that all
		coefficients of $w$ belong to $F$. If $w$ has no occurrence of $x$, then $w$
		is a non-trivial element of $\Gamma$, and there is nothing to prove. Thus we
		may assume that $w$ contains at least one occurrence of $x$.
		
		Conjugating $w$ inside $\Gamma * \langle x\rangle$ does not affect whether
		$w(\gamma)$ is trivial. We may therefore replace $w$ by a conjugate whose
		reduced normal form starts and ends with non-zero powers of $x$. Indeed, if
		the reduced normal form starts with $x^a$, $a\neq 0$, and ends with a
		coefficient, we conjugate by $x^{\operatorname{sgn}(a)}$. If it starts with a
		coefficient and ends with $x^b$, $b\neq 0$, we conjugate by
		$x^{-\operatorname{sgn}(b)}$. If it starts and ends with coefficients, we
		conjugate by $x^{-1}$. In each case, the resulting reduced word starts and
		ends with non-zero powers of $x$, and its coefficients still belong to $F$.
		
		Thus, after this replacement, we may write
		\[
		w=x^{m_0}f_1x^{m_1}f_2\cdots f_kx^{m_k},
		\]
		where $k\geq 0$, where $m_0,\ldots,m_k$ are non-zero integers, and where
		$f_1,\ldots,f_k\in F$.
		
		We evaluate $w(\gamma)$ at the point $p\in S\setminus U$, reading the word
		from right to left. The rightmost factor $\gamma^{m_k}$ sends $p$ into $U$.
		If $k\geq 1$, the next factor $f_k$ sends this point outside $U$, since
		$f_kU\cap U=\varnothing$. The next non-zero power of $\gamma$ sends it back
		into $U$. Continuing in this way, every coefficient $f_i$ sends the current
		point from $U$ to $S\setminus U$, and every non-zero power of $\gamma$ sends
		the current point from $S\setminus U$ back into $U$.
		
		Since the leftmost term of $w$ is a non-zero power of $x$, the final result
		lies in $U$. Hence $w(\gamma)p\in U$. But $p\in S\setminus U$, so
		$w(\gamma)p\neq p$, and therefore $w(\gamma)\neq e$.
	\end{proof}

	Corollaries~\ref{cor:A2-Pstar-full-limit} and~\ref{cor:C2-Pstar-full-limit} therefore yield a class of groups with the MIF property. Theorem~\ref{Thm:MIF} follows, since cocompact lattices are of general type and have full flag limit set by \cite[Proposition~6.1]{ciobotaru_le-bars26}.

    \subsection{Selfless reduced $C^*$-algebras}
transversal contractivity provides a single element with suitable contractive dynamics. The following lemma shows that this can be bootstrapped to arbitrarily many such elements whose associated dynamics are mutually transverse.

	\begin{lem}\label{lem:transversal-for-n=1}
Let \(\Gamma\) be an infinite group acting on a non-empty set \(S\). The
action is transversely contractive if and only if, for every finite set
\(F\subseteq\Gamma\setminus\{e\}\) and every \(n\in\mathbb N\), there exist
elements \(\gamma_1,\ldots,\gamma_n\in\Gamma\) and pairwise disjoint
subsets
\(
	U_1^+,U_1^-,\ldots,U_n^+,U_n^-\subseteq S
	\)
such that:
\begin{enumerate}
\item for every \(i=1,\ldots,n\),
\(
	\gamma_i(S\setminus U_i^-)\subseteq U_i^+;
	\)
\item for every \(f\in F\),
\(
	\left(\bigcup_{i=1}^n(U_i^+\cup U_i^-)\right)
	\cap
	f\left(\bigcup_{i=1}^n(U_i^+\cup U_i^-)\right)
	=\varnothing.
	\)
\end{enumerate}
\end{lem}

\begin{proof}
The implication from the displayed condition to transversal contractivity
is immediate by taking \(n=1\).

Conversely, assume that the action is transversely contractive. Fix
\(n\in\mathbb N\) and a finite set
\(F\subseteq\Gamma\setminus\{e\}\). Since \(\Gamma\) is infinite, we may
choose distinct elements \(h_1,\ldots,h_n\in\Gamma\) such that
\(h_j^{-1}fh_i\neq e\) for all \(f\in F\) and all \(i,j\). This can be
done recursively, avoiding only finitely many forbidden values at each
step.

Set
\[
L_0:=
\{h_j^{-1}h_i\mid i\neq j\}
\cup
\{h_j^{-1}fh_i\mid f\in F,\ 1\leq i,j\leq n\}.
\]
By transversal contractivity, there exist \(\gamma\in\Gamma\) and
disjoint subsets \(U^+,U^-\subseteq S\) such that, writing
\(U:=U^+\cup U^-\), we have
\(\gamma(S\setminus U^-)\subseteq U^+\) and
\(\ell U\cap U=\varnothing\) for every \(\ell\in L_0\).

Define \(\gamma_i:=h_i\gamma h_i^{-1}\) and
\(U_i^\pm:=h_iU^\pm\). Then
\(\gamma_i(S\setminus U_i^-)\subseteq U_i^+\). If \(i\neq j\), the sets
\(U_i^\sigma\) and \(U_j^\tau\) are disjoint for all
\(\sigma,\tau\in\{+,-\}\), since
\((h_j^{-1}h_i)U\cap U=\varnothing\); for \(i=j\), their disjointness
follows from \(U^+\cap U^-=\varnothing\).

Finally, for \(f\in F\), we have
\(fU_i^\sigma\cap U_j^\tau=\varnothing\) for all \(i,j\) and
\(\sigma,\tau\in\{+,-\}\), since
\((h_j^{-1}fh_i)U\cap U=\varnothing\). Hence the union of all the
\(U_i^\pm\) is disjoint from each of its \(F\)-translates, as required.
\end{proof}

Having the preceding lemma in mind, selflessness follows from Ozawa's criterion.\footnote{An earlier version of \cite{Ozawa2025} contained a formulation of the criterion particularly close to transversal contractivity.}
\begin{prop}
If \(\Gamma\) is transversely contractive, then \(C_r^*(\Gamma)\) is  selfless.
\end{prop}

\begin{proof}
By Lemma~\ref{lem:transversal-for-n=1}, for every finite set
\(F\subseteq\Gamma\setminus\{e\}\) and every \(n\in\mathbb N\), there exist
\(\gamma_1,\ldots,\gamma_n\in\Gamma\) and pairwise disjoint subsets
\(U_1^+,U_1^-,\ldots,U_n^+,U_n^-\) satisfying the contraction and
transversality conditions of Definition~\ref{def:P-star-PHP}. Thus, after
choosing \(x_0\in\Gamma\) and setting
\(C_i=D_i:=\{s\in\Gamma:sx_0\in U_i^+\}\), these data satisfy
\(P_{PHP}\) in the sense of \cite{Ozawa2025}. Hence
\(C_r^*(\Gamma)\) is completely selfless by \cite[Theorem~14]{Ozawa2025}.
\end{proof}

Corollaries~\ref{cor:A2-Pstar-full-limit} and~\ref{cor:C2-Pstar-full-limit} therefore yield a class of groups whose reduced $C^*$-algebras are selfless. Theorem~\ref{Thm:MIF} follows, since cocompact lattices are of general type and have full flag limit set by \cite[Proposition~6.1]{ciobotaru_le-bars26}.

	\small
	\bibliographystyle{alpha}
	\bibliography{Refs}

@article{Tomanov1985,
  author  = {Tomanov, G. M.},
  title   = {Generalized group identities in linear groups},
  journal = {Mathematics of the USSR-Sbornik},
  volume  = {51},
  number  = {1},
  pages   = {33--46},
  year    = {1985},
  doi     = {10.1070/SM1985v051n01ABEH002845}
}

@article{HullOsin2016,
  author  = {Hull, Michael and Osin, Denis},
  title   = {Transitivity degrees of countable groups and acylindrical hyperbolicity},
  journal = {Israel Journal of Mathematics},
  volume  = {216},
  number  = {1},
  pages   = {307--353},
  year    = {2016},
  doi     = {10.1007/s11856-016-1411-9}
}

@article{Jacobson2021,
  author  = {Jacobson, Bryan},
  title   = {A mixed identity-free elementary amenable group},
  journal = {Communications in Algebra},
  volume  = {49},
  number  = {1},
  pages   = {235--241},
  year    = {2021},
  doi     = {10.1080/00927872.2020.1797073}
}

@article{EtedadialiabadiGaoLeMaitreMelleray2021,
  author  = {Etedadialiabadi, Mahmood and Gao, Su and Le Ma{\^\i}tre, Fran{\c{c}}ois and Melleray, Julien},
  title   = {Dense locally finite subgroups of automorphism groups of ultraextensive spaces},
  journal = {Advances in Mathematics},
  volume  = {391},
  pages   = {107966},
  year    = {2021},
  doi     = {10.1016/j.aim.2021.107966}
}

@article{BodirskySchneiderThom2025,
  author  = {Bodirsky, Manuel and Schneider, Jakob and Thom, Andreas},
  title   = {Mixed identities for oligomorphic automorphism groups},
  journal = {The Journal of Symbolic Logic},
  year    = {2025},
  pages   = {1--23},
  doi     = {10.1017/jsl.2025.10123},
  note    = {Published online 19 August 2025}
}

@article{Bradford2024,
  author  = {Bradford, Henry},
  title   = {Quantifying lawlessness in finitely generated groups},
  journal = {Journal of Group Theory},
  volume  = {27},
  number  = {1},
  pages   = {31--59},
  year    = {2024},
  doi     = {10.1515/jgth-2022-0113}
}

@article{BradfordSchneiderThom2024,
  author  = {Bradford, Henry and Schneider, Jakob and Thom, Andreas},
  title   = {On the length of nonsolutions to equations with constants in some linear groups},
  journal = {Bulletin of the London Mathematical Society},
  volume  = {56},
  number  = {7},
  pages   = {2338--2349},
  year    = {2024},
  doi     = {10.1112/blms.13058}
}

@misc{AvniGelander2025,
  author        = {Avni, Nir and Gelander, Tsachik},
  title         = {Mixed identities in linear groups -- effective version},
  year          = {2025},
  eprint        = {2510.03492},
  archivePrefix = {arXiv},
  primaryClass  = {math.GR}
}

@article{BradfordSisto2026,
  author  = {Bradford, Henry and Sisto, Alessandro},
  title   = {Non-solutions to mixed equations in acylindrically hyperbolic groups coming from random walks},
  journal = {Archiv der Mathematik},
  volume  = {126},
  number  = {4},
  pages   = {343--350},
  year    = {2026},
  doi     = {10.1007/s00013-026-02223-4}
}

@misc{HydeLodha2025,
  author        = {Hyde, James and Lodha, Yash},
  title         = {Embeddings into highly transitive and mixed identity free groups},
  year          = {2025},
  eprint        = {2509.09788},
  archivePrefix = {arXiv},
  primaryClass  = {math.GR}
}

@misc{Rybak2026,
  author        = {Rybak, Ekaterina},
  title         = {Boundary dynamics, triple transitivity, and mixed identities in weakly hyperbolic groups},
  year          = {2026},
  eprint        = {2605.14159},
  archivePrefix = {arXiv},
  primaryClass  = {math.GR}
}

@article{Oppenheim2025,
	author  = {Oppenheim, Izhar},
	title   = {Property~{(T)} for Groups Acting on Affine Buildings},
	journal = {Bulletin of the London Mathematical Society},
	volume  = {57},
	number  = {10},
	pages   = {3151--3162},
	year    = {2025},
	doi     = {10.1112/blms.70148}
}

@article{KleinerLeeb1997,
	author  = {Kleiner, Bruce and Leeb, Bernhard},
	title   = {Rigidity of Quasi-Isometries for Symmetric Spaces and
	{E}uclidean Buildings},
	journal = {Publications Math\'ematiques de l'IH\'ES},
	volume  = {86},
	pages   = {115--197},
	year    = {1997},
	doi     = {10.1007/BF02698902}
}

@misc{TitzMiteWitzel2025,
	author = {Thomas Titz Mite and Stefan Witzel},
	title = {Non-residually finite {$\tilde{C}_2$}-lattices},
	year = {2025},
	howpublished = {Preprint, {arXiv}:2509.05054},
	url = {https://arxiv.org/abs/2509.05054},
	arXiv = {arXiv:2509.05054}
}

@misc{lecureux_witzel26,
	author = {Jean L{\'e}cureux and Stefan Witzel},
	title = {The {Normal} {Subgroup} {Theorem} for lattices on two-dimensional {Euclidean} buildings},
	year = {2026},
	howpublished = {Preprint, {arXiv}:2605.06163},
	url = {https://arxiv.org/abs/2605.06163},
	arXiv = {arXiv:2605.06163}
}

@misc{Ozawa2025,
	author = {Narutaka Ozawa},
	title = {Proximality and selflessness for group {C}*-algebras},
	year = {2026},
	howpublished = {Preprint, {arXiv}:2508.07938},
	url = {https://arxiv.org/abs/2508.07938},
	arXiv = {arXiv:2508.07938}
}

@incollection{deLaHarpe1985,
	author    = {de la Harpe, Pierre},
	title     = {Reduced {$C^*$}-Algebras of Discrete Groups Which Are
	Simple with a Unique Trace},
	booktitle = {Operator Algebras and Their Connections with Topology
	and Ergodic Theory},
	series    = {Lecture Notes in Mathematics},
	volume    = {1132},
	pages     = {230--253},
	publisher = {Springer-Verlag},
	address   = {Berlin},
	year      = {1985}
}

@incollection {parreau00,
	AUTHOR = {Parreau, Anne},
	TITLE = {Immeubles affines: construction par les normes et \'{e}tude des
	isom\'{e}tries},
	BOOKTITLE = {Crystallographic groups and their generalizations ({K}ortrijk,
	1999)},
	SERIES = {Contemp. Math.},
	VOLUME = {262},
	PAGES = {263--302},
	PUBLISHER = {Amer. Math. Soc., Providence, RI},
	YEAR = {2000},
	MRCLASS = {20E42 (20G25 51E24 53C23)},
	MRNUMBER = {1796138},
	MRREVIEWER = {Guy Rousseau},
	DOI = {10.1090/conm/262/04180},
	URL = {https://doi-org.ezproxy.universite-paris-saclay.fr/10.1090/conm/262/04180},
}

@article {robert25,
	AUTHOR = {Robert, Leonel},
	TITLE = {Selfless {${\rm C}^*$}-algebras},
	JOURNAL = {Adv. Math.},
	FJOURNAL = {Advances in Mathematics},
	VOLUME = {478},
	YEAR = {2025},
	PAGES = {Paper No. 110409, 28},
	ISSN = {0001-8708,1090-2082},
	MRCLASS = {46L05 (46L35 46L54)},
	MRNUMBER = {4924062},
	MRREVIEWER = {Bipul\ Saurabh},
	DOI = {10.1016/j.aim.2025.110409},
	URL = {https://doi-org.ezproxy.weizmann.ac.il/10.1016/j.aim.2025.110409},
}

@article{schafhauser2025nuclear,
  title={Nuclear {$C^*$}-algebras: 99 problems},
  author={Schafhauser, Christopher and Tikuisis, Aaron and White, Stuart},
  journal={arXiv preprint arXiv:2506.10902},
  year={2025}
}

@misc{BFFHZ26,
	author        = {Belk, James and Fournier-Facio, Francesco and
	Hyde, James and Zaremsky, Matthew C. B.},
	title         = {Boone--Higman Embeddings of
	{$\operatorname{Aut}(F_n)$} and Mapping Class Groups
	of Punctured Surfaces},
	year          = {2026},
	eprint        = {2503.21882},
	archivePrefix = {arXiv},
	primaryClass  = {math.GR},
	note          = {arXiv:2503.21882, Revised July 2026}
}

@article{Vigdorovich2026Linear,
	author = {Vigdorovich, Itamar},
	title = {Selfless reduced ${C^{*}}$-algebras of linear groups},
	journal = {Proceedings of the London Mathematical Society},
	volume = {133},
	number = {1},
	pages = {e70180},
	doi = {https://doi.org/10.1112/plms.70180},
	url = {https://londmathsoc.onlinelibrary.wiley.com/doi/abs/10.1112/plms.70180},
	eprint = {https://londmathsoc.onlinelibrary.wiley.com/doi/pdf/10.1112/plms.70180},
	year = {2026}
}

@misc{GaoKunnawalkamPatchellTeryoshin2026,
  author        = {Gao, David and
                   Kunnawalkam Elayavalli, Srivatsav and
                   Patchell, Gregory and
                   Teryoshin, Lizzy},
  title         = {Selfless reduced amalgamated free products and
                   {HNN} extensions},
  year          = {2026},
  eprint        = {2604.06982},
  archivePrefix = {arXiv},
  primaryClass  = {math.OA}
}

@misc{ArzhantsevaFinnSell2026,
  author        = {Arzhantseva, Goulnara and Finn-Sell, Martin},
  title         = {Lacunary hyperbolic groups with fast injectivity radius
                   growth and enough loxodromic elements are selfless},
  year          = {2026},
  eprint        = {2606.20456},
  archivePrefix = {arXiv},
  primaryClass  = {math.GR}
}

@misc{BasuFlores2026,
  author        = {Basu, Aaratrick and Flores, Felipe},
  title         = {Selfless inclusions arising from commensurator groups
                   of hyperbolic groups},
  year          = {2026},
  eprint        = {2605.12737},
  archivePrefix = {arXiv},
  primaryClass  = {math.GR}
}

@misc{Yang2025,
  author        = {Yang, Wenyuan},
  title         = {An extreme boundary of acylindrically hyperbolic groups},
  year          = {2025},
  eprint        = {2511.16400},
  archivePrefix = {arXiv},
  primaryClass  = {math.GR}
}

@misc{KunnawalkamPatchellTeryoshin2025,
  author        = {Kunnawalkam Elayavalli, Srivatsav and
                   Patchell, Gregory and
                   Teryoshin, Lizzy},
  title         = {Some remarks on decay in countable groups and
                   amalgamated free products},
  year          = {2025},
  eprint        = {2509.08754},
  archivePrefix = {arXiv},
  primaryClass  = {math.GR}
}

@article {benoist97,
	AUTHOR = {Benoist, Y.},
	TITLE = {Propri\'et\'es asymptotiques des groupes lin\'eaires},
	JOURNAL = {Geom. Funct. Anal.},
	FJOURNAL = {Geometric and Functional Analysis},
	VOLUME = {7},
	YEAR = {1997},
	NUMBER = {1},
	PAGES = {1--47},
	ISSN = {1016-443X,1420-8970},
	MRCLASS = {22E15},
	MRNUMBER = {1437472},
	MRREVIEWER = {Scot\ Adams},
	DOI = {10.1007/PL00001613},
	URL = {https://doi.org/10.1007/PL00001613},
}

@article {ciobotaru_muhlerr_rousseau_20,
	AUTHOR = {Ciobotaru, Corina and M\"{u}hlherr, Bernhard and Rousseau, Guy},
	TITLE = {The cone topology on masures},
	NOTE = {With an appendix by Auguste H\'{e}bert},
	JOURNAL = {Adv. Geom.},
	FJOURNAL = {Advances in Geometry},
	VOLUME = {20},
	YEAR = {2020},
	NUMBER = {1},
	PAGES = {1--28},
	ISSN = {1615-715X},
	MRCLASS = {20E42 (20G44 51E24)},
	MRNUMBER = {4052945},
	MRREVIEWER = {Walter D. Freyn},
	DOI = {10.1515/advgeom-2019-0020},
	URL = {https://doi-org.ezproxy.universite-paris-saclay.fr/10.1515/advgeom-2019-0020},
}

@book {bridson_haefliger99,
	AUTHOR = {Bridson, Martin R. and Haefliger, Andr\'{e}},
	TITLE = {Metric spaces of non-positive curvature},
	SERIES = {Grundlehren der mathematischen Wissenschaften [Fundamental
	Principles of Mathematical Sciences]},
	VOLUME = {319},
	PUBLISHER = {Springer-Verlag, Berlin},
	YEAR = {1999},
	ISBN = {3-540-64324-9},
	MRCLASS = {53C23 (20F65 53C70 57M07)},
	MRNUMBER = {1744486},
	MRREVIEWER = {Athanase Papadopoulos},
	DOI = {10.1007/978-3-662-12494-9},
	URL = {https://doi.org/10.1007/978-3-662-12494-9},
}

@book{kramer90,
	author = {Kramer, Linus},
	title = {Compact polygons},
	year = {1994},
	publisher = {T{\"u}bingen: Math. Fak., Univ. T{\"u}bingen},
	language = {English},
	zbMATH = {815387},
	Zbl = {0844.51006}
}

@article {kapovich_leeb_porti18morse,
	AUTHOR = {Kapovich, Michael and Leeb, Bernhard and Porti, Joan},
	TITLE = {A {M}orse lemma for quasigeodesics in symmetric spaces and
	{E}uclidean buildings},
	JOURNAL = {Geom. Topol.},
	FJOURNAL = {Geometry \& Topology},
	VOLUME = {22},
	YEAR = {2018},
	NUMBER = {7},
	PAGES = {3827--3923},
	ISSN = {1465-3060,1364-0380},
	MRCLASS = {53C23 (20E42 20F65 51E24)},
	MRNUMBER = {3890767},
	MRREVIEWER = {Guy\ Rousseau},
	DOI = {10.2140/gt.2018.22.3827},
	URL = {https://doi.org/10.2140/gt.2018.22.3827},
}

@article {kapovich_leeb_porti17anosov,
	AUTHOR = {Kapovich, Michael and Leeb, Bernhard and Porti, Joan},
	TITLE = {Anosov subgroups: dynamical and geometric characterizations},
	JOURNAL = {Eur. J. Math.},
	FJOURNAL = {European Journal of Mathematics},
	VOLUME = {3},
	YEAR = {2017},
	NUMBER = {4},
	PAGES = {808--898},
	ISSN = {2199-675X,2199-6768},
	MRCLASS = {22E40 (20F65 53C35)},
	MRNUMBER = {3736790},
	MRREVIEWER = {Herbert\ Abels},
	DOI = {10.1007/s40879-017-0192-y},
	URL = {https://doi.org/10.1007/s40879-017-0192-y},
}

@article {kramer02,
    AUTHOR = {Kramer, Linus},
     TITLE = {Loop groups and twin buildings},
      NOTE = {Dedicated to John Stallings on the occasion of his 65th
              birthday},
   JOURNAL = {Geom. Dedicata},
  FJOURNAL = {Geometriae Dedicata},
    VOLUME = {92},
      YEAR = {2002},
     PAGES = {145--178},
      ISSN = {0046-5755,1572-9168},
   MRCLASS = {20E42 (51E24 51H15 53C42)},
  MRNUMBER = {1934016},
MRREVIEWER = {Guy\ Rousseau},
       DOI = {10.1023/A:1019603827308},
       URL = {https://doi.org/10.1023/A:1019603827308},
}

@book {rousseau23,
	AUTHOR = {Rousseau, Guy},
	TITLE = {Euclidean buildings---geometry and group actions},
	SERIES = {EMS Tracts in Mathematics},
	VOLUME = {35},
	PUBLISHER = {EMS Press, Berlin},
	YEAR = {[2023] \copyright 2023},
	PAGES = {x+597},
	ISBN = {978-3-98547-039-6; 978-3-98547-539-1},
	MRCLASS = {51E24 (20C08 20E42 20F55)},
	MRNUMBER = {4632266},
	DOI = {10.4171/etm/35},
	URL = {https://doi.org/10.4171/etm/35},
}

@article {amrutam_gao_kunnawalkam-elayavalli_patchell25,
	AUTHOR = {Amrutam, Tattwamasi and Gao, David and Kunnawalkam Elayavalli,
	Srivatsav and Patchell, Gregory},
	TITLE = {Strict comparison in reduced group {$C^*$}-algebras},
	JOURNAL = {Invent. Math.},
	FJOURNAL = {Inventiones Mathematicae},
	VOLUME = {242},
	YEAR = {2025},
	NUMBER = {3},
	PAGES = {639--657},
	ISSN = {0020-9910,1432-1297},
	MRCLASS = {46L05 (46L35)},
	MRNUMBER = {4978176},
	MRREVIEWER = {Qingzhai\ Fan},
	DOI = {10.1007/s00222-025-01366-5},
	URL = {https://doi-org.ezproxy.weizmann.ac.il/10.1007/s00222-025-01366-5},
}

@misc{vigdorovich25,
	author = {Itamar Vigdorovich},
	title = {Structural properties of reduced {$C^*$}-algebras associated with higher-rank lattices},
	year = {2025},
	howpublished = {Preprint, {arXiv}:2503.12737},
	url = {https://arxiv.org/abs/2503.12737},
	arXiv = {arXiv:2503.12737}
}

@article {parkinson_van-maldeghem24,
	AUTHOR = {Parkinson, James and Van Maldeghem, Hendrik},
	TITLE = {Automorphisms and opposition in spherical buildings of
	classical type},
	JOURNAL = {Adv. Geom.},
	FJOURNAL = {Advances in Geometry},
	VOLUME = {24},
	YEAR = {2024},
	NUMBER = {3},
	PAGES = {287--321},
	ISSN = {1615-715X,1615-7168},
	MRCLASS = {20E42 (20E45 51A50 51E24)},
	MRNUMBER = {4781463},
	MRREVIEWER = {Theo\ Grundh\"ofer},
	DOI = {10.1515/advgeom-2024-0012},
	URL = {https://doi-org.ezproxy.weizmann.ac.il/10.1515/advgeom-2024-0012},
}

@article {parkinson_van-maldeghem19,
	AUTHOR = {Parkinson, James and Van Maldeghem, Hendrik},
	TITLE = {Opposition diagrams for automorphisms of large spherical
	buildings},
	JOURNAL = {J. Combin. Theory Ser. A},
	FJOURNAL = {Journal of Combinatorial Theory. Series A},
	VOLUME = {162},
	YEAR = {2019},
	PAGES = {118--166},
	ISSN = {0097-3165,1096-0899},
	MRCLASS = {51E24 (20E42)},
	MRNUMBER = {3873873},
	MRREVIEWER = {Rupert\ McCallum},
	DOI = {10.1016/j.jcta.2018.09.011},
	URL = {https://doi-org.ezproxy.weizmann.ac.il/10.1016/j.jcta.2018.09.011},
}

@article {bader_caprace_lecureux19,
	AUTHOR = {Bader, Uri and Caprace, Pierre-Emmanuel and L\'{e}cureux, Jean},
	TITLE = {On the linearity of lattices in affine buildings and
	ergodicity of the singular {C}artan flow},
	JOURNAL = {J. Amer. Math. Soc.},
	FJOURNAL = {Journal of the American Mathematical Society},
	VOLUME = {32},
	YEAR = {2019},
	NUMBER = {2},
	PAGES = {491--562},
	ISSN = {0894-0347},
	MRCLASS = {20E42 (20C99 20E08 20F65 22D40 22E40 22F50 51E24)},
	MRNUMBER = {3904159},
	MRREVIEWER = {Herbert Abels},
	DOI = {10.1090/jams/914},
	URL = {https://doi-org.ezproxy.universite-paris-saclay.fr/10.1090/jams/914},
}

@article{`2023selfless,
	title={Selfless {$C^*$}-algebras},
	author={Robert, Leonel},
	journal={arXiv preprint arXiv:2309.14188},
	year={2023}
}

@article{CL,
	title={Stability of Homomorphisms, Coverings and Cocycles II: Examples, Applications and Open problems},
	author={Chapman, M. and Lubotzky, A.},
	journal={arXiv preprint arXiv:2311.06706},
	year={2023}
}

@article {caprace_ciobotaru15,
	AUTHOR = {Caprace, Pierre-Emmanuel and Ciobotaru, Corina},
	TITLE = {Gelfand pairs and strong transitivity for {E}uclidean
	buildings},
	JOURNAL = {Ergodic Theory Dynam. Systems},
	FJOURNAL = {Ergodic Theory and Dynamical Systems},
	VOLUME = {35},
	YEAR = {2015},
	NUMBER = {4},
	PAGES = {1056--1078},
	ISSN = {0143-3857,1469-4417},
	MRCLASS = {51E24 (22D15)},
	MRNUMBER = {3345164},
	MRREVIEWER = {Linus\ K. H. Kramer},
	DOI = {10.1017/etds.2013.102},
	URL = {https://doi.org/10.1017/etds.2013.102},
}

@article {parkinson_temmermans_van-maldeghem15,
	AUTHOR = {Parkinson, James and Temmermans, Beukje and Van Maldeghem,
	Hendrik},
	TITLE = {The combinatorics of automorphisms and opposition in
	generalised polygons},
	JOURNAL = {Ann. Comb.},
	FJOURNAL = {Annals of Combinatorics},
	VOLUME = {19},
	YEAR = {2015},
	NUMBER = {3},
	PAGES = {567--619},
	ISSN = {0218-0006,0219-3094},
	MRCLASS = {51E12 (05B25)},
	MRNUMBER = {3395495},
	MRREVIEWER = {Paul-Hermann\ Zieschang},
	DOI = {10.1007/s00026-015-0277-6},
	URL = {https://doi-org.ezproxy.weizmann.ac.il/10.1007/s00026-015-0277-6},
}

@book {abramenko_brown08,
	AUTHOR = {Abramenko, Peter and Brown, Kenneth S.},
	TITLE = {Buildings},
	SERIES = {Graduate Texts in Mathematics},
	VOLUME = {248},
	NOTE = {Theory and applications},
	PUBLISHER = {Springer, New York},
	YEAR = {2008},
	PAGES = {xxii+747},
	ISBN = {978-0-387-78834-0},
	MRCLASS = {20E42 (20F55 20J06 51E24 51F15)},
	MRNUMBER = {2439729},
	MRREVIEWER = {Ralf Koehl},
	DOI = {10.1007/978-0-387-78835-7},
	URL = {https://doi-org.ezproxy.universite-paris-saclay.fr/10.1007/978-0-387-78835-7},
}

@book{leeb00,
	AUTHOR = {Leeb, Bernhard},
	TITLE = {A characterization of irreducible symmetric spaces and
	{E}uclidean buildings of higher rank by their asymptotic
	geometry},
	SERIES = {Bonner Mathematische Schriften [Bonn Mathematical
	Publications]},
	VOLUME = {326},
	PUBLISHER = {Universit\"at Bonn, Mathematisches Institut, Bonn},
	YEAR = {2000},
	PAGES = {ii+42},
	MRCLASS = {53C24 (51E24)},
	MRNUMBER = {1934160},
	MRREVIEWER = {Ralf\ J.\ Spatzier},
}

@misc{ciobotaru_le-bars26,
	author = {Corina Ciobotaru and Corentin Le Bars},
	title = {Dynamical boundaries of affine buildings: {C}*-simplicity and {Poisson} boundaries},
	year = {2026},
	howpublished = {Preprint, {arXiv}:2601.13092},
	url = {https://arxiv.org/abs/2601.13092},
	arXiv = {arXiv:2601.13092}
}

@incollection {kapovich_leeb18,
	AUTHOR = {Kapovich, Michael and Leeb, Bernhard},
	TITLE = {Discrete isometry groups of symmetric spaces},
	BOOKTITLE = {Handbook of group actions. {V}ol. {IV}},
	SERIES = {Adv. Lect. Math. (ALM)},
	VOLUME = {41},
	PAGES = {191--290},
	PUBLISHER = {Int. Press, Somerville, MA},
	YEAR = {2018},
	ISBN = {978-1-57146-365-4},
	MRCLASS = {22E40 (20E42 20F65 53C35)},
	MRNUMBER = {3888689},
	MRREVIEWER = {Jean\ Raimbault},
}

@article{GrundhoferKramerVanMaldeghemWeiss2012,
	author = {Grundh{\"o}fer, Theo and Kramer, Linus and Van Maldeghem, Hendrik and Weiss, Richard M.},
	title = {Compact totally disconnected {Moufang} buildings},
	journal = {Tohoku Mathematical Journal},
	volume = {64},
	number = {3},
	pages = {333--360},
	year = {2012},
	doi = {10.2748/tmj/1348863660}
}

@article{temmermans_thas_vanmaldeghem12,
	author = {Temmermans, Beukje and Thas, Joseph A. and Van Maldeghem, Hendrik},
	title = {Domesticity in generalized quadrangles},
	journal = {Annals of Combinatorics},
	volume = {16},
	number = {4},
	pages = {905--916},
	year = {2012},
	doi = {10.1007/s00026-012-0145-6}
}

@incollection {ronan86,
	AUTHOR = {Ronan, Mark},
	TITLE = {A construction of buildings with no rank {$3$} residues of
	spherical type},
	BOOKTITLE = {Buildings and the geometry of diagrams ({C}omo, 1984)},
	SERIES = {Lecture Notes in Math.},
	VOLUME = {1181},
	PAGES = {242--248},
	PUBLISHER = {Springer, Berlin},
	YEAR = {1986},
	ISBN = {3-540-16466-9},
	MRCLASS = {51D20 (51E25)},
	MRNUMBER = {843395},
	MRREVIEWER = {S.\ V.\ Tsaranov},
	DOI = {10.1007/BFb0075518},
	URL = {https://doi.org/10.1007/BFb0075518},
}

@article{flores2025selfless,
  title = {Selfless reduced free products and graph products of {$C^*$}-algebras},
  author = {Flores, Felipe and Klisse, Mario and Cobhthaigh, M{\'\i}che{\'a}l {\'O} and Pagliero, Matteo},
  journal = {arXiv preprint arXiv:2510.24675},
  year = {2025},
}

	\vspace{0.5cm}

	\noindent{\textsc{D\'epartement de Math\'ematiques et Applications, \'Ecole Normale Sup\'erieure - PSL, 45 rue d'Ulm, 75005 Paris, France}}
	\vspace{0.2cm}
	
	\noindent{\textit{Email address:} \texttt{corentin.le.bars@ens.psl.eu}} \\

	\noindent{\textsc{Department of Mathematics, Computer Science and Statistics, Ghent University, Krijgslaan 281-S9, 9000 Gent, Belgium}}
	\vspace{0.2cm}
	
	\noindent{\textit{Email address:} \texttt{elyasheev.leibtag@gent.be}} \\

	\noindent{\textsc{Mathematical Institute, University of Oxford, Radcliffe Observatory, Andrew Wiles Building, Woodstock Rd, Oxford OX2 6GG, United Kingdom}}
	\vspace{0.2cm}
	
	\noindent{\textit{Email address:} \texttt{itamar.vigdorovich@maths.ox.ac.uk}} \\

\end{document}